\documentclass{article}
\usepackage{graphicx} 
\usepackage{amsmath,amssymb,amsfonts}
\usepackage{amsthm}
\usepackage{mathtools}

\usepackage{mathtools}
\usepackage{hyperref}
\usepackage{enumitem}

\usepackage[normalem]{ulem}

\usepackage[margin=1in]{geometry}

\usepackage{makecell}
\usepackage{cite}

\usepackage{todonotes}

\usepackage{cancel}

\usepackage{comment}

\newcommand{\norm}[2]{\left\Vert #1 \right\Vert_{#2}}

\newcommand{\dv}[0]{\mathrm{div} \, }
\newcommand{\dvh}[0]{\mathrm{div}_{\mathrm{h}} \, }
\newcommand{\mrm}[1]{\mathrm{#1}}
\newcommand{\nablah}[0]{\nabla_\mathrm{h}}

\newcommand{\dt}[0]{\partial_t}

\newcommand{\dz}[0]{\partial_z}

\renewcommand{\vec}[1]{\mathbf{#1}}

\theoremstyle{plain}
\newtheorem{theorem}{Theorem}[section]

\newtheorem{lemma}[theorem]{Lemma}
\newtheorem{proposition}[theorem]{Proposition}

\theoremstyle{remark}
\newtheorem{remark}{Remark}[section]

\theoremstyle{definition}
\newtheorem{definition}{Definition}[section]

\newcommand{\divh}{\mathrm{div}_h}

\numberwithin{equation}{section}

\allowdisplaybreaks

\title{Rigorous justification of the hydrostatic-incompressible approximation for weakly stratified isothermal flow}
\author{Quyuan Lin\footnote{\href{mailto:quyuanl@clemson.edu}{quyuanl@clemson.edu}, School of Mathematical and Statistical Sciences, Clemson University, Clemson, SC 29634, USA}, 
Xin Liu\footnote{\href{mailto:xliu23@tamu.edu}{xliu23@tamu.edu}, Department of Mathematics, Texas A\&M University, College Station, TX 77843-3368, USA}}
\date{\today}

\begin{document}

\maketitle

\begin{abstract}
    We consider the limit of small Mach number and small vertical-to-horizontal aspect ratio for the isothermal compressible Navier-Stokes system. In addition, we consider the scale in which the stratification is weak. Owing to the anisotropic nature of the problem, the dynamics exhibit a three-wave separation phenomenon, consistent of a slow wave, a fast horizontal acoustic wave, and an even faster vertical acoustic wave. These three waves, unfortunately, are not mutually orthogonal, and the corresponding projections are parametrized by the small parameter, which significantly complicates the nonlinear analysis. Without any restriction on the size of the initial waves, we establish the uniform existence and uniqueness of solutions to the compressible Navier-Stokes system for any fixed small parameter, by carefully analyzing the evolutions of both the energy and the acoustic waves. Moreover, we prove that, as the small parameter tends to zero, the solutions converge to that of the incompressible primitive equations governing atmospheric and oceanic flows. 

    \medskip


\end{abstract}

\tableofcontents

\section{Introduction}
\label{sec:introduction}

In the large scale geophysical flows, including flows in the atmosphere, the ocean, and the Great Lakes, 
the vertical scale is much smaller than the horizontal scale. Moreover, the most interesting and relevant scales in these geophysical flows usually admit small Mach and Froude numbers. See Section \ref{subsec:typical_values}, below. 

In this paper, we are interested in the physical scale for the ocean and the Great Lakes (see Table \ref{tb:typical_values}, below). To be more precise, we consider the following dimensionless isothermal Navier-Stokes equations (NSEs):
\begin{subequations}
    \label{sys:hst-bsnq-CNS-intro}
    \begin{gather}
        \label{hbCNS-continuity-intro}
        \partial_t \rho + \dvh (\rho \vec v) + \partial_z (\rho w) = 0, \\
        \label{hbCNS-h-momentum-intro}
        \begin{gathered}
        \dt (\rho \vec v) + \dvh (\rho \vec v\otimes \vec v) + \dz (\rho w \vec v) + \frac{1}{\varepsilon^2} \nablah p(\rho) 
        = \Delta \vec v,
        \end{gathered} \\
        \label{hbCNS-v-momentum-intro}
        \begin{gathered}
        \dt(\rho w) + \dvh(\rho w \vec v) + \dz(\rho w w ) + \frac{1}{\varepsilon^4} \dz \rho = - \frac{1}{\varepsilon^3} \rho \dz g(z) 
        + \Delta w,
    \end{gathered}
    \end{gather}
\end{subequations}
where the aspect ratio, Mach number, and Froude number are assumed to be of the same small order $\varepsilon\ll1$. Here $ \rho $ is the density, $ p(\rho) = \rho $ is the isothermal pressure potential, $ \vec v $ and $ w $ are the horizontal and the vertical velocity fields, and $ g = g(z) $ is the gravity potential. See section \ref{sec:nndmsnlsn}, below for the nondimensionalization.

Our goal is to rigorously justify the asymptotic limit of system \eqref{sys:hst-bsnq-CNS-intro}, as $\varepsilon\to 0$. The limit system will be shown to be the incompressible primitive equations (PEs)
\begin{subequations}
    \label{sys:hst-bsnq-intro}
    \begin{gather}
        \label{eq:continuity-hb-intro}
        \dvh \vec v_p + \dz w_p = 0, \\
        \label{eq:h-momentum-hb-intro}
        \partial_t \vec v_p + \vec v_p \cdot \nablah \vec v_p + w_p \dz \vec v_p + \nablah \Pi_p = \Delta v_p,\\
        \label{eq:hstc-hb-intro}
        \dz \Pi_p = 0, 
    \end{gather}
\end{subequations}
where $\vec v_p$, $w_p$, and $\Pi_p$ are the horizontal velocity, vertical velocity, and pressure, respectively.

\smallskip

The incompressible PEs, also known as the hydrostatic NSEs, have been widely used as an approximation model in the study of oceanic dynamics (see, e.g., \cite{blumen1972geostrophic,gill1976adjustment,gill1982atmosphere,hermann1993energetics,holton1973introduction,kuo1997time,plougonven2005lagrangian,rossby1938mutual} and references therein). System \eqref{sys:hst-bsnq-intro}
has been rigorously
derived as the asymptotic limit of the small aspect ratio from the incompressible NSEs in the weak sense \cite{azerad2001mathematical} and in the strong sense 
\cite{li2019primitive}. See also \cite{furukawa2020rigorous} for the case with different initial conditions and \cite{li2022primitive} for the case with only horizontal viscosity. The rigorous derivation of incompressible PEs with temperature evolution from Boussinesq equations remains open. Unlike the three-dimensional NSEs of which the global well-posedness is unknown, the incompressible PEs are globally well-posed in 3D. The latter was first established in the pioneer work  \cite{Cao2007}. See also \cite{kobelkov2006existence} for an alternative approach, \cite{kukavica2007regularity} for different boundary conditions, as well as \cite{hieber2016global} for some progress towards relaxing the smoothness on the initial data by using the maximum regularity and the semigroup method.

The inviscid counterpart of the incompressible PEs have been derived from the incompressible Euler equations with initial data satisfying the local Rayleigh condition or being analytic \cite{brenier1999homogeneous,brenier2003remarks,grenier1999derivation,masmoudi2012h}. Unlike the viscous version, the inviscid system is in general ill-posed in Sobolev spaces \cite{renardy2009ill,han2016ill,ibrahim2021finite}, and its smooth solution can form singularity in finite time \cite{cao2015finite,collot2023stable,ibrahim2026profile,wong2015blowup,ibrahim2021finite}. These strongly suggest the necessity of viscosity in system \eqref{sys:hst-bsnq-CNS-intro} in order to prove the convergence for general initial data.

\smallskip 

For the compressible isothermal Navier--Stokes equations \eqref{sys:hst-bsnq-CNS-intro}, global weak solutions with small initial data were constructed in \cite{hoffGlobalSolutionsNavierStokes1995a}. In contrast, without a smallness assumption, the classical Lions--Feireisl theory for isentropic flows \cite{Lions1998,Feireisl2001} is not available in the isothermal case. The local well-posedness of system \eqref{sys:hst-bsnq-CNS-intro} follows from the classical works \cite{Hoff2012,Itaya1971}.

The study of singular limits in compressible flows was initiated by pioneering works in the incompressible/small Mach number limit \cite{Klainerman1981,Klainerman1982} for well-prepared initial data, namely data close to the incompressible regime. For general initial data in the whole space, acoustic waves disperse, as shown in \cite{Ukai1986}. In periodic domains, however, acoustic waves may resonate with themselves. Nevertheless, the mean flow can be separated from the acoustic component in the nonlinear dynamics, as established in \cite{Lions1998a,Masmoudi2001}. Consequently, in these settings, the incompressible system provides a valid approximation of the full compressible system in the small Mach number regime. In contrast, the small Mach number limit in bounded domains with nontrivial acoustic waves remains an open problem. 

When gravity is taken into account, the resulting stratification makes the asymptotic limit even less clear. For weak solutions with well-prepared data, the incompressible limit under both weak and strong stratification was established in \cite{feireislSingularLimitsThermodynamics2017}. In the absence of acoustic waves, the Boussinesq approximation, namely the incompressible limit with weak stratification, was justified in \cite{rajagopalOberbeckboussinesqApproximation1996}. These justifications for general initial data remain open. 

The hydrostatic limit for compressible isothermal flows was rigorously justified in \cite{liuRigorousJustificationHydrostatic2024}, where the compressible primitive equations arise as the limit system.
The latter was shown to be locally well-posed \cite{liuLocalWellPosednessStrong2021} and to admit global weak solutions with degenerate viscosity \cite{LT2018b}. The small Mach/incompressible limit of the compressible primitive equations was established in \cite{liuZeroMachNumber2023,LT2018LowMach1} for well-prepared and ill-prepared initial data. We also refer to \cite{necasovaEnergyEqualityCompressible2024,tangDerivationInviscidCompressible2023,Ersoy2012,Ersoy2011a,hieberLagrangianApproachCompressible2025,gaoHydrostaticApproximationCompressible2022} for related works on the compressible primitive equations.

\smallskip 

For general systems of PDEs with fast oscillation, the first general theory was developed by Schochet in \cite{Schochet1994} for two-scales systems; see also \cite{schochetSingularLimitsBounded1987,SchochetCMP1986,Gallagher1998} for further developments. Schochet's theory has been applied to singular limits in fluid dynamics, especially to geophysical flows, in \cite{embidAveragingFastGravity1996,majdaAveragingFastGravity1998,embidLowFroudeNumber1998,MajdaAtmosphereOcean}, to mention a few.
More recently, increasing attention has been devoted to systems involving three distinct scales; see, for instance, \cite{schochetUniformExistenceConvergence2022,chengThreeScaleSingularLimits2018,schochetModeratelyFastThreeScale2020}. The systems considered in these works are of the form
\begin{equation}
\label{sys:3-scales-system}
    A^0(\varepsilon U) \partial_t U + \frac{1}{\varepsilon}\mathcal L U + \frac{1}{\delta} \mathcal M U + N(U) = 0,
\end{equation}
where $A^0(0)$, $\mathcal L$, and $\mathcal M$ are linear operators with constant coefficients, while $N(U)$ denotes the nonlinear terms. Under suitable restrictions on the parameters $\varepsilon$ and $\delta$, together with appropriate prepared initial data, these works establish uniform estimates and identify the corresponding asymptotic limits. Unfortunately, our system \eqref{sys:hst-bsnq-CNS-intro} (equivalently, \eqref{sys:hst-bsnq-sym}, below) is not of the form \eqref{sys:3-scales-system}, and therefore, should be investigated independently. 

\smallskip

Meanwhile in \cite{Klein2010,breschSoundproofModelAcoustic2022}, assuming that $ \delta = \delta(\varepsilon) $ and for the modified compressible Euler system with fast acoustic and internal waves, the authors established the uniform existence and the asymptotic limit without any restriction on the initial data. Their analysis relies crucially on a clear separation of the slow, fast, and faster wave components. Moreover, the projection operators onto these wave modes are regular: They are bounded operators without loss of regularity in any component. However, such a regular decomposition is no longer available for system \eqref{sys:hst-bsnq-CNS-intro}; see Definition \ref{def:wave-proj} below. This is caused by the anisotropic structure induced by the small aspect ratio. In particular, the very fast vertical acoustic wave carries a singular source term on the right-hand side; see \eqref{wdpt-105} below.

Moreover, the anisotropic structure of system \eqref{sys:hst-bsnq-CNS-intro} leads to an energy estimate that degenerates in the vertical velocity $w$. In contrast to \cite{li2019primitive,li2022primitive}, here the incompressibility condition is not available to recover uniform regularity of $w$. We therefore follow the approach of \cite{liuRigorousJustificationHydrostatic2024}. By studying the associated acoustic wave equation and exploiting the continuity equation \eqref{hbCNS-continuity-intro}, we recover the required uniform regularity of $w$. Another main difference from \cite{liuRigorousJustificationHydrostatic2024,liuZeroMachNumber2023} is the presence of gravity in system \eqref{sys:hst-bsnq-CNS-intro}. As shown in Section \ref{subsec:non-prd-non-cnvg} below, gravity may in general cause instability in the limit $\varepsilon \rightarrow 0$. We address this issue by introducing an artificial gravity potential satisfying $\dz g = \mathcal O(\varepsilon)$ and by developing a weighted energy estimate that incorporates the effect of stratification; see Section \ref{subsec:ln-energy-stuct} and remarks \ref{rm:loss-regularity}, \ref{rm:weights}, below. 

\smallskip

We postpone the precise statements of the main results to Theorem~\ref{thm:uniform-est} and Theorem~\ref{thm:limit} in Section~\ref{subsec:reformulation} below. 
To conclude the introduction, we state the main results informally as follows.

\begin{theorem}[Uniform stability]
    \label{informal-thm:uniform-stability}
    Without any restriction on the size of the initial slow, fast, and very fast waves (see Remark \ref{rm:initial_data}, below), system \eqref{sys:hst-bsnq-CNS-intro} admits a unique strong solutions $ (\rho, \vec v, w) $ on a time interval independent of $\varepsilon$. Moreover, the solution satisfies a uniform-in-$\varepsilon$ regularity bound. 
\end{theorem}

\begin{theorem}[Asymptotic limit]
    \label{informal-thm:convergence}
    Under the assumptions of Theorem~\ref{informal-thm:uniform-stability}, the slow dynamics of system \eqref{sys:hst-bsnq-CNS-intro} is characterized, as $\varepsilon \rightarrow 0$, by the limit system \eqref{sys:hst-bsnq-intro}. In particular, the solution converges in a suitable topology to $(1,\vec v_p,w_p)$, where $ (v_p, w_p) $ is a solution to system \eqref{sys:hst-bsnq-intro}.
\end{theorem}

We would like to comment that our theorems do not hold for isentropic flow; see Remark \ref{rm:isentropic-obstruction}, below. The case of isentropic flow is still open and left to a future work. 

The rest of the paper is organized as follows. 
In Section~\ref{sec:formal-limit}, we present the formal asymptotic limit for the reader's convenience, in the spirit of \cite{Klein2006}. In particular, the dimensionless system \eqref{sys:hst-bsnq-CNS-intro} is derived in Section~\ref{sec:nndmsnlsn}.  In Section~\ref{sec:rfm-ef-la}, we reformulate system \eqref{sys:hst-bsnq-CNS-intro} and carry out the linear analysis. In particular, the three-scales wave structure is investigated in Section~\ref{subsec:linear}. The key ingredients for the uniform estimates are introduced in Sections~\ref{subsec:ln-energy-stuct} and \ref{subsec:ln-acoustic-wave}, where we study the energy evolution and the acoustic wave evolution, respectively. Theorems~\ref{informal-thm:uniform-stability} and \ref{informal-thm:convergence} are proved in Sections~\ref{sec:uniform-est} and \ref{sec:limit}, respectively. Finally, in Section~\ref{sec:non-cnvg}, we discuss several other regimes in which the linear analysis suggests that the asymptotic limit may fail.

\section{Formal asymptotic limit}
\label{sec:formal-limit}

Let $ \rho = \rho(\vec x,t), \ p = p(\rho) \in \mathbb R^+\cup \lbrace 0 \rbrace $ be the (nonnegative) density and the pressure potential, and $ \vec u = \vec u(\vec x,t) = (\vec v, w)(\vec x,t) \in \mathbb R^2 \times \mathbb R $ be the velocity field. Here $ \vec v $ and $ w $ are the horizontal and the vertical velocities, respectively. 
Then a compressible, viscous flow, under the influence of a generalized gravity, is governed by the following compressible Navier-Stokes system:
\begin{subequations}\label{sys:CNS}
    \begin{gather}
        \label{CNS:continuity}
        \partial_t \rho + \dv (\rho \vec u) = 0,\\
        \label{CNS:momentum}
        \partial_t (\rho \vec u) + \dv (\rho \vec u \otimes \vec u) + \nabla p = - \rho \nabla g(z) + \dv \mathbb S. 
    \end{gather}
Here 
$ g = g(z)$ is the given generalized gravity potential, and $ \mathbb S $ is the viscosity tensor given by
\begin{equation}
    \label{def:viscosity-tensor}
    \mathbb S:= \left(\begin{array}{cc}
    \mu_{hh} (\nablah \vec v + \nablah \vec v^\top) + \lambda \dv \vec u \mathbb I_2 & \mu_{hz} (\nablah w + \partial_z \vec v) \\
    \mu_{zh} (\nablah^\top w + \partial_z \vec v^\top) & 2 \mu_{zz} \partial_{z} w  + \lambda \dv \vec u
    \end{array}\right),
\end{equation}
where $ \mu_{hh}, \mu_{zh} = \mu_{hz}, \mu_{zz}$, and $ \lambda $ are nonnegative constants. 
\end{subequations}

We consider the system in a domain $\mathcal D:=(L\mathbb T)^2 \times [0,H]$ with periodic boundary conditions in the horizontal variables and 
\begin{equation}
    \label{bc} 
    (w,\partial_z \vec v)|_{z=0,H}=0. 
\end{equation} 
Here $ L $ and $ H $ are the typical length scales in the horizontal and vertical directions, respectively. Notably, for $ g(z) = g z $, system \eqref{sys:CNS} is reduced to the standard compressible Navier-Stokes system with homogeneous gravity.

\subsection{Nondimensionalization}\label{sec:nndmsnlsn}

Following \cite{Klein2006}, 
we introduce the following scaling of system \eqref{sys:CNS}  
\begin{equation}
    \label{def:scaling}
    \begin{aligned}
        \vec x_\mrm{h} =& L \vec x_\mrm h', & z =& H z', & t =& T t', & g =& GH g', \\
        \rho =& D \rho', & \vec v =& V \vec v', & w =& W w', & p =& P p',
    \end{aligned}
\end{equation}
where $ L, H, T, G, D, V, W, P >0 $ are the typical values of the corresponding variables. 
We further assume and denote that
\begin{equation}
    \label{asmp:scaling}
    \begin{aligned}
    \sigma := & \frac{H}{L} = \frac{W}{V}, & L = & VT, & H = & WT, \\
    \mathrm{Ma} :=& \frac{V}{\sqrt{P/D}}, & \mrm{Fr} :=& \sqrt{\frac{WV}{GH}} = \frac{V}{\sqrt{GL}}, & \mrm{Re}_d := & \frac{DVL}{\lambda}, \\
    \mrm{Re}_{hh} := & \frac{D VL}{\mu_{hh}}, & \mrm{Re}_{hz} := & \frac{D VL}{\mu_{hz}}, & \mrm{Re}_{zz} := & \frac{D VL}{\mu_{zz}} .
    \end{aligned}
\end{equation}
Then after substituting \eqref{def:scaling} and \eqref{asmp:scaling} into \eqref{sys:CNS}, and writing the resultant equations without the primes, one obtains the following nondimensionalized system:
\begin{subequations}
    \label{sys:CNS-scaled}
    \begin{gather}
        \label{sCNS:continuity}
        \partial_t \rho + \dvh (\rho \vec v) + \dz (\rho w) = 0, \\
        \label{sCNS:h-momentum}
        \begin{gathered}
        \partial_t (\rho \vec v) + \dvh(\rho \vec v \otimes \vec v) + \dz (\rho w \vec v) + \frac{1}{\mrm{Ma}^2} \nablah p(\rho)\\
        = \frac{1}{\mrm{Re}_{hh}} [\dvh(\nablah \vec v + \nablah \vec v^\top) ] + \frac{1}{\mrm{Re}_d}\nablah (\dvh \vec v + \dz w)\\
        + \frac{1}{\sigma^2 \mrm{Re}_{hz}} \dz[ \sigma^2 \nablah w + \dz \vec v ],
        \end{gathered}\\
        \label{sCNS:v-momentum}
        \begin{gathered}
            \partial_t (\rho w) + \dvh ( \rho w \vec v) + \dz (\rho w w) + \frac{1}{\sigma^2 \mrm{Ma}^2} \dz p(\rho) = - \frac{1}{\sigma \mrm{Fr}^2}\rho \partial_z g(z) \\
            + \frac{1}{\sigma^2 \mrm{Re}_{hz}} \dvh[ \sigma^2 \nablah w + \dz \vec v ] + \frac{1}{\sigma^2 \mrm{Re}_{zz}} 2 \partial_{zz} w \\ 
            + \frac{1}{\sigma^2 \mrm{Re}_{d}} \dz ( \dvh \vec v + \dz w).
        \end{gathered}
    \end{gather}
\end{subequations}
Notice that we have defined the Mach and Froude numbers in \eqref{asmp:scaling} using the horizontal (speed and length) scales. One can easily using the vertical scales to define them and write down the resulting system of equations. However, for our application, it is easier to measure the horizontal scales of large-scale atmospheric and oceanic flows. We therefore adopt the definitions in \eqref{asmp:scaling} throughout this work.

\subsection{Typical values and isothermal flow}
\label{subsec:typical_values}
From \cite{USATMO},  \cite{OceanDepth}, and \cite{GreatLakes}, one can find the typical values of the relevant physical quantities in Table \ref{tb:typical_values}.
\begin{table}[h] 
\begin{center}
\begin{tabular}{|c|c|c|}
    \hline
    Property & Notation & Typical value \\
    \hline\hline
    Sound speed$^{2}$ & $ P/D $ & \makecell{$ 10^5 m^2 s^{-2} $ (Troposphere),\\ $ 10^6 m^2 s^{-2} $ (Water)}  \\
    \hline
    Density & $ D $ & \makecell{$10^3 g m^{-3} $ (Troposphere), \\ $ 10^6 g m^{-3} $ (Water)} \\
    \hline
    \makecell{Width \\ (horizontal scale)} & $ L $ & \makecell{$ 10^6 \sim 10^7 m $ (Troposphere and ocean), \\ $ 10^5 \sim 10^6 m $ (Great lakes) } \\
    \hline
    \makecell{Height/Depth \\ (vertical scale)} & $ H $ & \makecell{ $ 10^4 m $ (Troposphere), \\ $ 10^3 m $ (Ocean), $ 10^2 m $ (Great lakes)} \\
    \hline 
    Horizontal velocity & $V = \frac{L}{T}$  & \makecell{ $ 1 \sim 10 ms^{-1} $ (Troposphere), \\
    $1 ms^{-1}$ (Ocean, Great lakes) }
    \\
    \hline
    Gravity constant & $ G $ & $ 10 m s^{-2} $ \\
    \hline
\end{tabular}
\caption{Typical values}
\label{tb:typical_values}
\end{center}
\end{table}
Together with \eqref{asmp:scaling}, one can check that 
\begin{equation}
\label{non-dimensional-number}
    \begin{aligned}
    \sigma  \simeq & \ 10^{-3} \sim 10^{-2} \ (\text{Troposphere}), 
    \ 10^{-4} \sim 10^{-3} \ (\text{Ocean, Great Lakes}), \\
    \mrm{Ma}^2 \simeq & \ 10^{-5} \sim 10^{-3} \ (\text{Troposphere}), \ 10^{-6} \ (\text{Ocean, Great Lakes}) , \\
    \mrm{Fr}^2 \simeq & \ 10^{-8} \sim 10^{-5} \ (\text{Troposphere}), \ 10^{-8} \sim 10^{-7}\ (\text{Ocean}),
    \ 10^{-7} \sim 10^{-6} \ (\text{Great Lakes}), \\
     \mrm{Ma}^2/\mrm{Fr}^2 \simeq & \ 10^2 \sim 10^3  \ (\text{Troposphere}), \ 10 \sim 10^2 \ (\text{Ocean}), 
     \ 1 \sim 10 \ (\text{Great Lakes}).
\end{aligned}
\end{equation}
In particular, one has that 
\begin{equation}
    \label{non-dimensional-number-2}
    \begin{aligned}
    \frac{\sigma^2 \mrm{Ma}^2}{\sigma \mrm{Fr}^2} \simeq & \ 1  \ (\text{Troposphere}), \ 10^{-2} \ (\text{Ocean}), \ 10^{-3} \ ( \text{Great Lakes} ).
    \end{aligned}
\end{equation}
From \eqref{sCNS:v-momentum}, one can see that
\begin{enumerate}[label = (\roman{*}), ref = (\roman{*})]
    \item when $ \sigma^2 \mrm{Ma}^2 / \sigma \mrm{Fr}^2 \simeq 1 $, the gravity and the pressure are nearly balanced, leading to {\bf strong stratification} of the density; we refer to this regime as {\bf hydrostatic-anelastic scale}, which is relevant, for example, to tropospheric flows;
    \item when $ \sigma^2 \mrm{Ma}^2 / \sigma \mrm{Fr}^2 \ll 1 $, the gravity is weaker than the pressure and therefore the leading order of the density profile will be constant along the vertical direction, causing only {\bf weak stratification}; we refer this regime as the {\bf hydrostatic-Boussinesq scale}, which is relevant, for example, to flows in the ocean and Great Lakes.
\end{enumerate}
These terminologies are inspired by the classical Boussinesq and anelastic approximations \cite{feireislSingularLimitsThermodynamics2017}.

In this work, we focus on the \textbf{hydrostatic-Boussinesq scale}. The \textbf{hydrostatic-anelastic scale} will be addressed in a forthcoming work. In the next section, we formally derive the \textbf{hydrostatic-Boussinesq approximation} of the compressible Navier--Stokes equations \eqref{sys:CNS-scaled}.

Moreover, we consider only the isothermal case, and for the sake of presentation, we take
\begin{equation}\label{asmp:isothermal}
    p(\vec x, t) = \rho(\vec x, t) \qquad \text{for all } t\geq 0.
\end{equation}
Such a flow is referred to as the isothermal flow \cite{lionsMathematicalTopicsFluid1996,Lions1998}.

\subsection{The hydrostatic-Boussinesq approximation for isothermal flow}

In this section, we consider the regime $\mrm{Ma}/\mrm{Fr} \simeq 1$ in \eqref{non-dimensional-number}, namely, the Mach and Froude numbers are of the same order. We also assume that the aspect ratio $\sigma$ is comparable to both $\mrm{Ma}$ and $\mrm{Fr}$. In particular, this scaling is relevant to the Great Lakes and corresponds to the hydrostatic-Boussinesq scale.

\smallskip

To be more precise, let $ \varepsilon\in (0,1) $ denote a small (scale) number and consider
\begin{equation}
    \label{hst-bsnq-001}
    \sigma = \mrm{Ma} = \mrm{Fr} = \varepsilon,
\end{equation} 
and 
\begin{equation}
    \label{hst-bsnq-002}
    \mrm{Re}_{hh} =1, \qquad \mrm{Re}_d = \mrm{Re}_{hz} = \mrm{Re}_{zz} = \frac{1}{\sigma^2} = \varepsilon^{-2}.
\end{equation}
\begin{remark}
    We remark that the choice of \eqref{hst-bsnq-001} is for the sake of clear presentation, below. The more general case can be handled similarly. The choice of \eqref{hst-bsnq-002} is to guarantees the turbulence viscosity effect remains of order $ 1 $, although it is not determined by the typical values in Table \ref{tb:typical_values}. The case with singular or degenerate viscosity is left for future work. 
\end{remark}

With \eqref{hst-bsnq-001}, \eqref{hst-bsnq-002}, and the isothermal assumption \eqref{asmp:isothermal}, system \eqref{sys:CNS-scaled} can be written as 
\begin{subequations}
    \label{sys:hst-bsnq-CNS}
    \begin{gather}
        \label{hbCNS-continuity}
        \partial_t \rho + \dvh (\rho \vec v) + \partial_z (\rho w) = 0, \\
        \label{hbCNS-h-momentum}
        \begin{gathered}
        \dt (\rho \vec v) + \dvh (\rho \vec v\otimes \vec v) + \dz (\rho w \vec v) + \frac{1}{\varepsilon^2} \nablah \rho \\
        = \dvh (\nablah \vec v + \nablah \vec v^\top) +  \varepsilon^2 \nablah (\dvh \vec v + \partial_z w) + \partial_z \lbrack \varepsilon^2 \nablah w + \partial_z \vec v\rbrack,
        \end{gathered} \\
        \label{hbCNS-v-momentum}
        \begin{gathered}
        \dt(\rho w) + \dvh(\rho w \vec v) + \dz(\rho w w ) + \frac{1}{\varepsilon^4} \dz \rho = - \frac{1}{\varepsilon^3} \rho \dz g(z) \\
        + \dvh \lbrack \varepsilon^2 \nablah w + \dz \vec v\rbrack + 2 \partial_{zz} w + \dz (\dvh \vec v + \dz w).
    \end{gathered}
    \end{gather}
\end{subequations}
Moreover, thanks to the symmetry of the equations, we assume further that
\begin{equation*}\tag{SYM}
    \label{SYM}
    \begin{aligned}
    \vec v, \rho, g(z) \quad & \text{are even with respect to the $ z $-variable}, \\
    w \qquad & \text{is odd with respect to the $ z $-variable}.
\end{aligned}
\end{equation*}
We consider system \eqref{sys:hst-bsnq-CNS} on the periodic domain $ \Omega := \mathbb T^2 \times 2 \mathbb T $. Thanks to the symmetry condition \eqref{SYM}, we immediately obtain
\begin{equation}
    w,\ \partial_z \vec v = 0 \qquad \text{on} \ \mathbb T^2 \times \lbrace z = k, k \in \mathbb Z \rbrace.
\end{equation}
In particular, these boundary conditions are consistent with \eqref{bc}.

\bigskip 

Next, we derive the leading order system as $ \varepsilon \rightarrow 0 $ in system \eqref{sys:hst-bsnq-CNS}, with only slow waves. To this end, we write 
\begin{equation}
    \label{scl-xpnsn:ansatz-1}
    \psi(\vec x,t) = \sum_{j = 0}^{\infty} \varepsilon^{j} \psi^{(j)}(\vec x, t)\qquad \text{for} \ \psi \in \lbrace \rho, \vec v, w\rbrace. 
\end{equation}
Then expanding system \eqref{sys:hst-bsnq-CNS} leads to the following:
\begin{align}
    \label{epsilon--4}
    \mathcal O(\varepsilon^{-4}): && \partial_z \rho^{(0)} & = 0,\\
    \label{epsilon--3}
    \mathcal O(\varepsilon^{-3}): && \dz \rho^{(1)} + \rho^{(0)}\dz g(z) & = 0, \\
    \label{epsilon--2}
    \mathcal O(\varepsilon^{-2}): && \nablah \rho^{(0)} & = 0, \\
    \label{epsilon--2-2} && \dz\rho^{(2)} + \rho^{(1)} \dz g(z) & = 0, \\
    \label{epsilon--1}
    \mathcal O(\varepsilon^{-1}): && \nablah \rho^{(1)} & = 0, 
\end{align}
and the equations of $ \mathcal O(1) $, 
\begin{gather}
    \label{epsilon-0-rho}
    \dt \rho^{(0)} + \dvh (\rho^{(0)}\vec v^{(0)}) + \dz (\rho^{(0)}w^{(0)}) = 0, \\
    \label{epsilon-0-v}
\begin{gathered}
    \dt (\rho^{(0)}\vec v^{(0)}) + \dvh (\rho^{(0)}\vec v^{(0)}\otimes \vec v^{(0)}) + \dz (\rho^{(0)} w^{(0)} \vec v^{(0)})\\
    + \nablah  \rho^{(2)} 
    = \dvh (\nablah \vec v^{(0)} + \nablah {\vec v^{(0)\top}}) + \partial_{zz} \vec v^{(0)},
\end{gathered}
\end{gather}
where we only list the ones that lead to a closed system. 

Now we are ready to derive, formally, the leading order system. In fact, \eqref{epsilon--4}, \eqref{epsilon--2}, and the conservation of mass \eqref{epsilon-0-rho} imply that, without loss of generality, 
\begin{equation*}
    \rho^{(0)} \equiv 1 \qquad \text{and} \qquad \dvh \vec v^{(0)} + \dz w^{(0)} = 0.
\end{equation*}
Meanwhile, from \eqref{epsilon--3} and \eqref{epsilon--1}, one can calculate that 
\begin{equation*}
    \rho^{(1)}(z) = c(t) - g(z)
\end{equation*}
for some function $c(t)$. Integrating the continuity equation \eqref{hbCNS-continuity} at order $\mathcal O(\varepsilon)$ over the space, we get that $\int \rho^{(1)} d\vec x$ is a constant in time. Therefore, $c(t)=c$ is a constant independent of time.
Consequently, \eqref{epsilon--2-2} yields
\begin{equation*}
    \partial_z \pi_p := \partial_z \rho^{(2)} = \partial_z \lbrack \frac{1}{2} g^2(z) - c g(z)\rbrack. 
\end{equation*}
Without loss of generality, we will take $ c = 1/2 $.

To sum up, the leading order system of \eqref{sys:hst-bsnq-CNS}, i.e., the hydrostatic-Boussinesq approximation for isothermal flow is given by 
\begin{subequations}
    \label{sys:hst-bsnq}
    \begin{gather}
        \label{eq:continuity-hb}
        \dvh \vec v_p + \dz w_p = 0, \\
        \label{eq:h-momentum-hb}
        \partial_t \vec v_p + \vec v_p \cdot \nablah \vec v_p + w_p \dz \vec v_p + \nablah \pi_p = \dvh (\nablah \vec v_p + \nablah \vec v_p^\top) + \partial_{zz} \vec v_p,\\
        \label{eq:hstc-hb}
        \dz \pi_p =  \frac{1}{2} \dz \lbrack g^2(z) - g(z) \rbrack.
    \end{gather}
\end{subequations}
Here we have used $ (\vec v_p, w_p, \pi_p) $ to represent the leading order horizontal velocity, vertical velocity, and the effective hydrostatic pressure. 

Writing $ \Pi_p:= \pi_p - \frac{1}{2}(g^2(z) - g(z)) $, system \eqref{sys:hst-bsnq} is equivalent to the following system: 
\begin{subequations}
    \label{sys:hst-bsnq-1}
    \begin{gather}
        \label{eq:continuity-hb-1}
        \dvh \vec v_p + \dz w_p = 0, \\
        \label{eq:h-momentum-hb-1}
        \partial_t \vec v_p + \vec v_p \cdot \nablah \vec v_p + w_p \dz \vec v_p + \nablah \Pi_p = \dvh (\nablah \vec v_p + \nablah \vec v_p^\top) + \partial_{zz} \vec v_p,\\
        \label{eq:hstc-hb-1}
        \dz \Pi_p = 0.
    \end{gather}
\end{subequations}

\bigskip

The local well-posedness of both system \eqref{sys:hst-bsnq-CNS} and \eqref{sys:hst-bsnq} follows easily with Galerkin's scheme, which may depend on the initial data. In fact, following \cite{Cao2007}, one has global well-posedness of system \eqref{sys:hst-bsnq}.
The goal of this work is to establish the uniform-in-$\varepsilon $  existence theory for system \eqref{sys:hst-bsnq-CNS} and then to rigorously justify the formal approximation.

\section{Reformulation, linear wave analysis, and linear energy structure}
\label{sec:rfm-ef-la}

\subsection{Eddy viscosity, reformulation, and main theorems}
\label{subsec:reformulation}

For the analysis below, we replace the anisotropic viscous operators in systems \eqref{sys:hst-bsnq-CNS} and \eqref{sys:hst-bsnq-1} by the isotropic eddy viscosity operators $\Delta \vec v$ and $\Delta w$. This leads to system \eqref{sys:hst-bsnq-CNS-intro} and corresponding limit system \eqref{sys:hst-bsnq-intro} in section \ref{sec:introduction}. The obstruction to treating the original anisotropic viscosity is explained in Remark~\ref{rm:eddy-viscosity}.

Next, we reformulate system \eqref{sys:hst-bsnq-CNS-intro} into a symmetric system for any $ \varepsilon \in (0,1) $. Let 
\begin{equation}
    q:=g(z)+\frac1\varepsilon\log \rho, \qquad \text{i.e.} \quad \rho=e^{\varepsilon(q-g(z))}.
    \label{def:symmetric_variable}
\end{equation}
By virtue of \eqref{SYM}, we have that $q$ is even with respect to the $ z $-variable.
From system \eqref{sys:hst-bsnq-CNS-intro}, one can write down the following {symmetric system} of $ (q,\vec v, w) $:
\begin{subequations}
    \label{sys:hst-bsnq-sym}
    \begin{gather}
\label{hbsym-continuity}
\dt q + \vec v \cdot \nablah q +  w \dz q  + \frac{1}{\varepsilon} (\dvh \vec v + \dz w)  =   w \dz g,  \\
\label{hbsym-h-momentum}
\begin{gathered}
\dt \vec v + \vec v \cdot \nablah \vec v + w \dz \vec v + \frac{1}{\varepsilon} \nablah q 
= \frac{1}{\rho} \Delta \vec v,
\end{gathered}\\
\label{hbsym-v-momentum}
\begin{gathered}
    \varepsilon^2  ( \dt w + \vec v \cdot \nablah w + w \dz w) + \frac{1}{\varepsilon} \dz q 
    = \frac{\varepsilon^2}{\rho} \Delta w.
\end{gathered}
    \end{gather}
\end{subequations}
Notice that, system \eqref{sys:hst-bsnq-sym} is degenerate as $ \varepsilon \rightarrow 0 $. To be more specific, one can only obtain the energy estimate, uniformly-in-$\varepsilon $, for the variables $ (q,\vec v, \varepsilon w ) $. Thus, one will lose the regularity of $ w $ at the limit. To resolve this problem, following \cite{liuRigorousJustificationHydrostatic2024}, we write from \eqref{hbCNS-continuity-intro} that
\begin{equation}
    \label{eq:w}
    \begin{gathered}
    w(x,y,z,t) = - \frac{1}{\rho(x,y,z,t)}\int_0^z\bigl\lbrack\dt \rho(x,y,z',t) + \dvh (\rho  \vec v (x,y,z',t)\bigr\rbrack \,dz'\\
     = - \frac{1}{\rho(x,y,z,t)}\int_0^z \rho(x,y,z',t) \bigl\lbrack \varepsilon \dt q(x,y,z',t) + \varepsilon \vec v \cdot \nablah q(x,y,z',t) +  \dvh \vec v(x,y,z',t) \bigr\rbrack \,dz'.
    \end{gathered}
\end{equation}
\begin{remark}[Symmetry with respect to the singular operator]
    \label{rm:sym-operator}
    The singular operator in system \eqref{sys:hst-bsnq-sym}, i.e., the terms involving $ \frac{1}{\varepsilon} $, is symmetric. This makes it possible to analyze the wave decomposition. See section \ref{subsec:linear}, below. 
\end{remark}
\begin{remark}[Loss of regularity]
    \label{rm:loss-regularity}
The identity \eqref{eq:w} implies that, to recover the regularity of $ w $, uniformly-in-$ \varepsilon $, it costs one tangential derivative 
of $ q $ and $ \vec v $, i.e., $  \varepsilon \nablah q, \nablah \vec v $, and the temporal derivative of $ q $, i.e., $ \varepsilon \dt q $. This is a reflection of the limit hydrostatic equations \eqref{sys:hst-bsnq-intro}, where $ w_p $ is recovered in a similar fashion as in \cite{Cao2007}, etc., and has been used in the study of the singular limit problems \cite{feireislSingularLimitsThermodynamics2017,liuZeroMachNumber2023,liuRigorousJustificationHydrostatic2024}, etc.. While the regularity of $ \nablah \vec v $ can be closed from the viscosity in \eqref{hbsym-h-momentum}, the regularity of $ \varepsilon \dt q, \varepsilon \nablah q $ cannot be closed by the energy estimate using \eqref{sys:hst-bsnq-sym} due to the fact that \eqref{hbsym-continuity} is a first-order equation. We will use weighted estimate via integrating factor to overcome this difficulty. See Remark \ref{rm:weights}, below. 
Moreover, we will use system \eqref{sys:hst-bsnq-sym} to derive the singular waves decomposition and eventually the asymptotic limit $ \varepsilon \rightarrow 0 $. 
\end{remark}

\begin{remark}[Acoustic wave]
    \label{rm:acoustic-wave}
    The standard energy estimate alone is not enough to close the uniform estimate, since it is degenerate in $ w $. The latter is important to control the vertical transport $ w \dz $. To recover the regularity of $ w $, as mentioned in Remark \ref{rm:loss-regularity}, one will need to estimate $ \varepsilon \dt q $, etc.. This will be done through tracking the evolution of the acoustic wave. See sections \ref{subsec:ln-acoustic-wave} and \ref{subsec:ene-est}, below.  
\end{remark}

\begin{remark}[Eddy viscosity]\label{rm:eddy-viscosity}
The key difference between the viscosity in systems \eqref{sys:hst-bsnq-CNS-intro} and \eqref{sys:hst-bsnq-CNS} is that system \eqref{sys:hst-bsnq-CNS} is degenerate in the horizontal viscosity for the vertical momentum equation, i.e., $ \dvh(\varepsilon^2 \nablah w) $ in \eqref{hbCNS-v-momentum}. We refer to the viscosity in system \eqref{sys:hst-bsnq-CNS-intro} as the eddy viscosity. For technical reasons, we do not treat the viscosity in system \eqref{sys:hst-bsnq-CNS}. This is due to the degeneracy in the acoustic wave equations, resulting from the degenerate horizontal viscosity for the vertical momentum, cf. Section \ref{subsec:ln-acoustic-wave}.
\end{remark}

Next we state our main results.

\begin{theorem}[Uniform stability]
\label{thm:uniform-est}
Consider initial data $ (q(0), \vec v(0),w(0))=(q_0, \vec v_0, w_0) $ with $ (\vec v_0, w_0) \in H^3(\mathbb T^2 \times 2 \mathbb T) $ and $q_0\in H^4(\mathbb T^2 \times 2 \mathbb T)$, and $ g = g(z) \in C^\infty (2 \mathbb T) $, satisfying the symmetry \eqref{SYM} and $ \dz g, \partial_{zz} g, \partial_{zzz} g = \mathcal O(\varepsilon) $. In addition, we assume that 
\begin{equation}
    \label{thmineq:initial_data}
    \mathcal E_0 := \norm{q(0), \vec v(0), \varepsilon w(0)}{H^3}^2 + \norm{\varepsilon \dt q(0), \varepsilon \nabla^2 q(0), \frac{\dz q(0)}{\varepsilon} }{H^2}^2 \leq M < \infty 
\end{equation}
for some positive constant $M$ that is uniformly bounded in $\varepsilon$,
where $ \varepsilon \dt q(0) $ is given by \eqref{hbsym-continuity}, i.e.,
\begin{equation}
    \varepsilon \dt q(0) =  \varepsilon\lbrack w(0) \dz g - \vec v(0)\cdot \nablah q(0) - w(0) \dz q(0) \rbrack - \dvh \vec v(0) - \dz w(0).
\end{equation}
Then there exist $ \varepsilon_0 \in (0,1) $ small enough and $ C, T \in (0,\infty) $, independent of $ \varepsilon $ but may depend on $ \mathcal E_0 $, such that for all $ \varepsilon \in (0,\varepsilon_0) $, 
\begin{equation}
    \label{thmineq:uniform-estimate}
    \begin{gathered}
        \sup_{0\leq t \leq T} \biggl\lbrace  \norm{q(t), \vec v(t), \varepsilon w(t)}{H^3}^2 + \norm{\varepsilon \dt q(t),\varepsilon\nabla^2 q(t), \frac{\dz q(t)}{\varepsilon}, w(t), \dz w(t)}{H^2}^2 \biggr\rbrace \\
        + \int_0^T \biggl( \norm{\nabla \vec v(t), \varepsilon \nabla w(t), w(t), \dz w(t)}{H^3}^2 + \norm{\Delta q(t), \frac{\nabla \dz q(t)}{\varepsilon}, \varepsilon \nabla \dt q(t), \varepsilon^2 \dt w(t), \varepsilon \dt \vec v(t)}{H^2}^2 \biggr) \,dt \\
        < C \mathcal E_0 < \infty. 
    \end{gathered}
\end{equation}
\end{theorem}

\begin{remark}
\label{rm:initial_data}
    We emphasize that the requirement of $ \dz q(0) = \mathcal O(\varepsilon) $ for the initial data \eqref{thmineq:initial_data} is not a restriction on the initial data. Indeed, the linear analysis, in particular \eqref{sol:k-th-mode-vertical-acoustic} in section \ref{subsec:linear}, below, indicates that $ \dz q $ corresponds to the `density' component of the vertical acoustic wave, which has to be $ \mathcal O(\varepsilon) $ in order to have bounded vertical acoustic wave. Otherwise, the initial data will be singular as $ \varepsilon \rightarrow 0 $. Therefore, the condition \eqref{thmineq:initial_data} should not be seen as a restriction on the size of the initial data. 
\end{remark}

\begin{theorem}[Asymptotic limit with general data]
\label{thm:limit}
    Under the same assumption as in Theorem \ref{thm:uniform-est}, there exists $ (\mathsf v_\sigma, \mathsf w_\sigma ) \in C([0,T];H^2(\mathbb T^2 \times 2 \mathbb T))\cap L^2(0,T;H^3(\mathbb T^2 \times 2 \mathbb T)) $ such that, upon selecting a subsequence as $ \varepsilon\rightarrow 0 $,
    \begin{equation}
        \label{thmcvg:weak_cnvgt}
        \begin{aligned}
        \rho = e^{\varepsilon(q-g(z))} & \rightarrow 1 \qquad & \text{in} & \ L^\infty(0,T;H^3(\mathbb T^2 \times 2 \mathbb T)), \\
        \vec v, \ w  & \rightharpoonup \mathsf v_\sigma,\ \mathsf w_\sigma \qquad & \text{weakly in} & \ L^2(0,T;H^3(\mathbb T^2 \times 2 \mathbb T)), \\
        \vec v, \ w  & \overset{*}{\rightharpoonup} \mathsf v_\sigma,\ \mathsf w_\sigma \qquad & \text{weak$*$-ly in} & \ L^\infty(0,T;H^2(\mathbb T^2 \times 2 \mathbb T)). 
        \end{aligned}
    \end{equation}
    Moreover, $ (\mathsf v_\sigma, \mathsf w_\sigma) $ is a solution to the limit system \eqref{sys:hst-bsnq-intro}.    
\end{theorem}

\begin{remark}
    The compactness \eqref{thmcvg:weak_cnvgt} is not sufficient in general to verify that $ (\mathsf v_\sigma, \mathsf w_\sigma) $ is a solution to the limit system. We leave the detailed compactness to Proposition \ref{prop:compactness}, below. In particular, the three-wave decomposition will be discussed in details in Section \ref{subsec:wave-dcmpt}. 
\end{remark}

\begin{proof}[Proof of the main theorems]
The proofs of Theorems \ref{thm:uniform-est} and \ref{thm:limit} are presented in Sections \ref{subsec:uniform-est} and \ref{subsec:limit-sys}, respectively. 
\end{proof}

\subsection{Linear wave analysis, three-scales separation, and projection operators}
\label{subsec:linear}

To better understand the wave separation phenomenon in the singular limit $ \varepsilon \rightarrow 0 $ of system \eqref{sys:hst-bsnq-sym}, we consider the following linear system:
\begin{subequations}
    \label{sys:linear-waves}
    \begin{align}
        \label{eq:linear-01}
        \dt \eta &&  + \frac{1}{\varepsilon} (\dvh \phi_\vec v + \dz \phi_w) && = 0, \\
        \label{eq:linear-02}
        \dt \phi_\vec v && + \frac{1}{\varepsilon}\nablah \eta && = 0, \\
        \label{eq:linear-03}
        \varepsilon^2 \dt \phi_w && +\frac{1}{\varepsilon} \dz \eta && = 0. 
    \end{align}
\end{subequations}
For any $ \vec k = (\vec k_h, k_3) \in 2\pi \mathbb Z^2 \times \pi \mathbb Z $, we search the $ \vec k $-th mode of the linear system \eqref{sys:linear-waves}, i.e., $ (\eta,\phi_\vec v, \phi_w)(\vec x, t) = (\hat \eta_\vec k, \hat \phi_{\vec v, \vec k}, \hat \phi_{w,\vec k}) e^{i (\vec k \cdot \vec x + \omega_\vec k t)} $. Without loss of generality, we only discuss the case when $ \vec k \neq 0 $. 
It is then straightforward to verify that
\begin{equation}
    \label{eq:k-th-mode}
    \begin{aligned}
        \varepsilon \omega_\vec k \hat \eta_\vec k + (\vec k_h \cdot \hat \phi_{\vec v,\vec k} + k_3 \hat \phi_{w,\vec k}) & = 0, \\
        \varepsilon \omega_\vec k \hat \phi_{\vec v, \vec k} +  \hat \eta_{\vec k} \vec k_h & = 0, \\
        \varepsilon^3 \omega_\vec k \hat \phi_{w,\vec k} + \hat \eta_{\vec k} k_3 & =0.
    \end{aligned}
\end{equation}
One can directly calculate that
\begin{itemize}
    \item {Mean wave:} \begin{equation}\label{sol:k-th-mode-meanflow}
    \begin{gathered}
        \omega_\vec k = \omega_{\vec k,0} = 0, \ \begin{pmatrix}
        \hat \eta_{\vec k, 0} \\ \hat \phi_{\vec v, \vec k, 0} \\ \hat \phi_{w,\vec k, 0}
        \end{pmatrix} = \begin{pmatrix}
            0 \\ \hat \phi_{\vec v,\vec k, \sigma} \\ \hat \phi_{w,\vec k, \sigma}
        \end{pmatrix}, \quad  \text{where}  \\
        \ \vec k_h \cdot \hat \phi_{\vec v,\vec k, \sigma} + k_3 \hat \phi_{w,\vec k, \sigma} = 0 \ \text{is the divergence-free mode};
        \end{gathered}
    \end{equation}
    \item {Horizontal acoustic wave:} \begin{equation}
        \label{sol:k-th-mode-horizontal-acoustic}
    \begin{gathered}
          \  k_3 = 0, \ \vec k_h \neq 0, \  \omega_\vec k = \omega_{\vec k_h, 0,\pm} = \pm \frac{\vert \vec k_h \vert}{\varepsilon} , \\
         \begin{pmatrix}
            \hat \eta_{\vec k, \pm} \\ \hat \phi_{\vec v, \vec k, \pm} \\ \hat \phi_{w,\vec k, \pm} 
        \end{pmatrix}  = \begin{pmatrix}
            1 \\
            \mp \frac{\vec k_h}{\vert \vec k_h \vert} \\ 0
        \end{pmatrix};
    \end{gathered}\\
    \end{equation}
    \item {Vertical acoustic wave:} \begin{equation}
        \label{sol:k-th-mode-vertical-acoustic}
    \begin{gathered}
         \ k_3 \neq 0, \ \omega_\vec k = \omega_{\vec k, \pm} = \pm \frac{\sqrt{\varepsilon^2 \vec k_h^2 + k_3^2}}{\varepsilon^2}, \\
        \begin{pmatrix}
            \hat \eta_{\vec k, \pm} \\ \hat \phi_{\vec v, \vec k, \pm} \\ \hat \phi_{w, \vec k, \pm} 
        \end{pmatrix} = \begin{pmatrix}
            \varepsilon  \\ \mp \frac{\varepsilon^2 \vec k_h}{\sqrt{\varepsilon^2 \vec k_h^2 + k_3^2}} \\ \mp \frac{k_3}{\sqrt{\varepsilon^2 \vec k_h^2 + k_3^2}}
        \end{pmatrix} = \begin{pmatrix}
        0 \\ 0 \\ \mp \frac{k_3}{\vert k_3 \vert} 
    \end{pmatrix} + \mathcal O(\varepsilon).
    \end{gathered} 
    \end{equation}
\end{itemize}
Here we have chosen the scale of $ (\hat \eta_{\vec k}, \hat \phi_{\vec v, \vec k}, \hat \phi_{w, \vec k}) $ to be $ \mathcal O(1) $ for convenience. One can immediately see that there are three time-scales in the linear system \eqref{sys:linear-waves}: The mean flow, characterized by $ \omega_\vec k = 0 $, referred to as the slow wave, below; the horizontal acoustic wave, characterized by $ \omega_\vec k \simeq \mathcal O (1/\varepsilon) $, referred to as the fast wave, below; and the vertical acoustic wave, characterized by $ \omega_\vec k \simeq \mathcal O (1/\varepsilon^2) $, referred to as the very fast wave, below. 

\bigskip
This inspires us to introduce the projection operators into the three types of wave from above.
Let $ \overline\cdot $ and $ \widetilde{\cdot} $ be the barotropic and baroclinic projections, respectively: For any $ f:\mathbb T^2 \times 2 \mathbb T \mapsto \mathbb R $,
\begin{equation}
    \label{def:barotropic-baroclinic}
    \begin{gathered}
    \overline f(x,y) := \frac{1}{2} \int_0^2 f(x,y,z)\,dz, \qquad \widetilde{f}(x,y,z):= f - \overline f(x,y),\\
    \text{and} \qquad 
    m_f:= \frac{1}{2}\int_{\mathbb T^2 \times 2 \mathbb T} f(x,y,z)\,dxdydz.
    \end{gathered}
\end{equation}
Then we define the wave projections as follows:
\begin{definition}[Wave projections]\label{def:wave-proj}
    Consider any vector field $ (\eta, \phi_\vec v, \phi_w): (x,y,z) \in \mathbb T^2 \times 2 \mathbb T \mapsto \mathbb R \times \mathbb R^2 \times \mathbb R $. 
    Let $ \Phi_h: \mathbb T^2 \mapsto \mathbb R $ and $ \Phi_3: \mathbb T^2 \times 2 \mathbb T \mapsto \mathbb R $ be the horizontal and vertical potential functions, given by
    \begin{equation}
        \label{def:potential-fns}
        \begin{aligned}
        \Phi_h & = \Phi_h(\phi_\vec v):= \Delta_h^{-1} \dvh \overline \phi_\vec v, \\ \quad \Phi_3 & = \Phi_3^\varepsilon = \Phi_3^\varepsilon(\phi_\vec v, \phi_w):= (\varepsilon^2 \Delta_h + \partial_{zz})^{-1}(\dvh \widetilde{\phi}_\vec v+\dz \phi_w).
        \end{aligned}
    \end{equation}
    Then we define 
    \begin{itemize}
        \item Horizontal acoustic projection: 
        \begin{equation}
            \label{def:ha_projector}
        \begin{aligned}
            & \begin{pmatrix}
            \phi_{\vec v, ha} \\ \phi_{w, ha}
            \end{pmatrix} = P_\mrm{ha}(\phi_\vec v, \phi_w):= \begin{pmatrix}
                \nablah \Phi_h(\phi_\vec v) \\ 0
            \end{pmatrix}\quad \text{and} \quad \\ 
            & \mathcal P_\mrm{ha}(\eta, \phi_\vec v, \phi_w):= \begin{pmatrix}
                \overline \eta - m_\eta \\ P_\mrm{ha}(\phi_\vec v, \phi_w)
            \end{pmatrix};
        \end{aligned}
        \end{equation}

        \item Vertical acoustic projection:
        \begin{equation}
            \label{def:va-projector}
            \begin{aligned} 
                & \begin{pmatrix}
            \phi_{\vec v, va} \\ \phi_{w, va}
            \end{pmatrix} = P^\varepsilon_\mrm{va}(\phi_\vec v, \phi_w):= \begin{pmatrix}
                \varepsilon^2 \nablah \Phi^\varepsilon_3 \\ \dz \Phi^\varepsilon_3
            \end{pmatrix} \quad \text{and} \quad \\ 
            &\mathcal P^\varepsilon_\mrm{va}(\eta, \phi_\vec v, \phi_w):= \begin{pmatrix}
            \widetilde \eta \\ P^\varepsilon_\mrm{va}(\phi_\vec v, \phi_w)
            \end{pmatrix};
            \end{aligned}
        \end{equation}

        \item Slow wave projection:
        \begin{equation}
            \label{def:sw-projector}
            \begin{aligned}
            & \begin{pmatrix}
            \phi_{\vec v, \sigma} \\ \phi_{w, \sigma}
            \end{pmatrix} = P^\varepsilon_\sigma(\phi_\vec v, \phi_w):= \begin{pmatrix}
                \phi_\vec v \\ \phi_w
            \end{pmatrix} - P_\mrm{ha}(\phi_\vec v, \phi_w) - P^\varepsilon_\mrm{va}(\phi_\vec v, \phi_w) \\
            & \text{and} \quad 
             \mathcal P^\varepsilon_\sigma(\eta,\phi_\vec v, \phi_w):= \begin{pmatrix}
                \eta \\ \phi_\vec v \\ \phi_w
            \end{pmatrix} - \mathcal P_\mrm{ha}(\eta,\phi_\vec v, \phi_w) - \mathcal P^\varepsilon_\mrm{va}(\eta,\phi_\vec v, \phi_w).
        \end{aligned}
        \end{equation}
    
    \end{itemize}
\end{definition}

Using the notations in Definition \ref{def:wave-proj}, one can write down, from the linear system \eqref{sys:linear-waves}, the following evolutionary equations for the three types of waves:

\begin{itemize}
    \item 
{\par\noindent\bf Mean flow:}
\begin{equation}
    \label{eq:meanflow-l}
    \dvh \phi_{\vec v, \sigma} + \dz \phi_{w,\sigma} = 0, \quad \partial_t m_\eta= \partial_t \phi_{\vec v,\sigma} = \partial_t \phi_{w,\sigma} = 0;
\end{equation}

\item 
{\par\noindent\bf The horizontal acoustic wave:}
\begin{equation}
    \label{eq:ha-l}
    \begin{gathered}
        \partial_t (\overline{\eta} - m_\eta) + \frac{1}{\varepsilon} \Delta_h \Phi_h = 0, \\
        \partial_t \nablah \Phi_h + \frac{1}{\varepsilon} \nablah \overline{\eta} = 0,
    \end{gathered}\quad \text{or equivalently}, \quad (\partial_t + \frac{1}{\varepsilon}\mathcal L_\mrm{ha} )\begin{pmatrix} \overline \eta - m_\eta \\ \nablah \Phi_h \\ 0 \end{pmatrix}  = 0,
\end{equation}
with the conservative energy given by 
\begin{equation}
    \label{def:ha-l-energy}
    \norm{\overline \eta(t) - m_\eta, \nablah \Phi_h(t)}{L^2}^2.
\end{equation}
Here $ \mathcal L_\mrm{ha} $ is the horizontal acoustic operator given by
\begin{equation}
    \label{def:horizontal-aw-operator}
    \mathcal L_\mrm{ha} \begin{pmatrix} \overline \eta - m_\eta \\ \nablah \Phi_h \\ 0 \end{pmatrix} := \begin{pmatrix}
        \Delta_h \Phi_h \\ \nablah (\overline \eta - m_\eta) \\ 0
    \end{pmatrix},
\end{equation}
and it is easy to verify that $ \mathcal L_\mrm{ha} $ is an skew-adjoint operator acting on vectors of the form $  (\overline \eta - m_\eta, \nablah \Phi_h, 0 ) $ and therefore $ \mathcal L_\mrm{ha} $ only has non-zero pure imaginary eigenvalues; see \eqref{sol:k-th-mode-horizontal-acoustic}. In particular, 
\begin{equation}
\label{norm-presering-ha}e^{\frac{t}{\varepsilon}\mathcal L_\mrm{ha}} \qquad \text{is $ H^s$-norm-preserving and fast oscillating};
\end{equation}

\item
{\par\noindent\bf The vertical acoustic wave:} 
\begin{equation}
    \label{eq:va-l}
    \begin{gathered}
    \partial_t \widetilde \eta + \frac{1}{\varepsilon}(\varepsilon^2 \Delta_h \Phi_3 + \partial_{zz}\Phi_3) = 0,\\
    \varepsilon^2 \partial_t \nablah \Phi_3 + \frac{1}{\varepsilon} \nablah \widetilde \eta = 0,\\
    \varepsilon^2 \partial_t \dz \Phi_3 + \frac{1}{\varepsilon} \dz \widetilde{\eta} = 0,
\end{gathered} \qquad \text{or equivalently,} \qquad ( \partial_t + \frac{1}{\varepsilon^2}\mathcal L_\mrm{va}^\varepsilon) \begin{pmatrix}
    \frac{\widetilde\eta}{\varepsilon} \\ \varepsilon \nablah \Phi_3 \\ \dz \Phi_3
\end{pmatrix} = 0,
\end{equation}
with the conservative energy given by
\begin{equation}
    \label{def:va-l-energy}
    \norm{\frac{\widetilde \eta(t)}{\varepsilon},\varepsilon \nablah \Phi_3(t), \dz \Phi_3(t)}{L^2}^2.
\end{equation}
Here $ \mathcal L_\mrm{va}^\varepsilon $ is the vertical acoustic operator given by
\begin{equation}
    \label{def:vertical-aw-operator}
    \mathcal L_\mrm{va}^\varepsilon \begin{pmatrix}
    \frac{\widetilde\eta}{\varepsilon} \\ \varepsilon \nablah \Phi_3 \\ \dz \Phi_3
\end{pmatrix} = \begin{pmatrix}
    \varepsilon^2 \Delta_h \Phi_3 + \partial_{zz} \Phi_3 \\ \nablah \widetilde \eta \\ \dz ( \frac{\widetilde \eta}{\varepsilon} )
\end{pmatrix},
\end{equation}
and it is easy to verify that $ \mathcal L^\varepsilon_\mrm{va} $ is an skew-adjoint operator acting on vectors of the form $ (\frac{\widetilde\eta}{\varepsilon}, \varepsilon \nablah \Phi_3, \dz \Phi_3 )$ and that $ \mathcal L^\varepsilon_\mrm{va} $ only has non-zero pure imaginary eigenvalues with uniform-in-$\varepsilon$ lower bound; see, \eqref{sol:k-th-mode-vertical-acoustic}. In particular, 
\begin{equation}
\label{norm-presering-va}
e^{\frac{t}{\varepsilon^2}\mathcal L^\varepsilon_\mrm{va}} \qquad \text{is $ H^s$-norm-preserving and fast oscillating}.
\end{equation}

\end{itemize}

It is straightforward to check that
\begin{equation*}
    \eqref{sys:linear-waves} = \eqref{eq:meanflow-l} + \eqref{eq:ha-l} + \eqref{eq:va-l}.
\end{equation*}
Moreover, for the vertical acoustic wave to be non-singular as $ \varepsilon \rightarrow 0 $, from \eqref{eq:va-l} and \eqref{def:va-l-energy}, one need $ \widetilde{\eta} = \mathcal O(\varepsilon) $, which is consistent with \eqref{sol:k-th-mode-vertical-acoustic}. This is equivalent to require 
\begin{equation}
\label{small-vertical-eta}
    \mathcal O(\dz \eta) = \mathcal O(\dz \widetilde \eta) = \mathcal O(\varepsilon). 
\end{equation}
Moreover, thanks to \eqref{sol:k-th-mode-vertical-acoustic}, in particular the fact that the frequency of the vertical acoustic waves is $ \mathcal O(1/\varepsilon^2) $, we have that
\begin{equation}
    \label{fast-vertical-acoustic-wave}
    \mathcal O(\partial_t^2 \eta) \geq \mathcal O(\partial_t^2 \widetilde{\eta})= \mathcal O(\frac{1}{\varepsilon^2})^2 \times \mathcal O(\varepsilon) = \mathcal O(\frac{1}{\varepsilon^3}).
\end{equation}
Then from \eqref{eq:linear-01}, it is easy to see that
\begin{equation}
    \label{time-derivative-w}
    \mathcal O(\varepsilon \dt \phi_w)\geq \mathcal O(\varepsilon^2 \partial_t^2 \eta) \geq \mathcal O(\frac{1}{\varepsilon}).
\end{equation}

\subsection{Linear energy structure}
\label{subsec:ln-energy-stuct}

In this section, we assume the following boundedness and non-degenerate condition for the density:
\begin{equation}
    \label{ene:a-priori-rho}
    0 < \frac{1}{2}\underline \rho < \rho(\vec x,t) < 2 \overline \rho < \infty,
\end{equation}
where $\underline \rho$ and $\overline{\rho}$ are positive constants. We will discuss the linear energy structure of system \eqref{sys:hst-bsnq-sym}.

\smallskip 

Denote by $ \Lambda:=(\xi, V, W):\mathbb T^2 \times 2 \mathbb T \times \mathbb R^+ \mapsto \mathbb R \times \mathbb R^2 \times \mathbb R $.
Consider the evolutionary equations associated with system \eqref{sys:hst-bsnq-sym}:

\begin{subequations}
    \label{sys:evolutionary-op}
    \begin{gather}
        \label{eq:evl-op-1}
        \partial_t \xi  +  \vec v \cdot \nablah \xi +  w \dz \xi + \frac{1}{\varepsilon} (\dvh V + \dz W) - W \dz g = G_1,\\
        \label{eq:evl-op-2}
        \begin{gathered}
        \dt V + \vec v \cdot \nablah V + w \dz V + \frac{1}{\varepsilon} \nablah \xi 
        - \frac{1}{\rho} \Delta V
        = G_2,
        \end{gathered}\\
        \label{eq:evl-op-3}
        \begin{gathered}
            \varepsilon^2 (\dt W + \vec v \cdot \nablah W + w \dz W) + \frac{1}{\varepsilon} \dz \xi 
            - \frac{\varepsilon^2}{\rho} \Delta W = \varepsilon G_3,
        \end{gathered}
    \end{gather}
\end{subequations}
or, for the sake of convenience, 
\begin{equation}
    \label{sys:evl-op-shorthand}
    \mathcal L(t,\rho, \vec v, w, \varepsilon)\Lambda = \mathcal G:= \begin{pmatrix}
        G_1 \\ G_2 \\ \varepsilon G_3
    \end{pmatrix},
\end{equation}
where $ \mathcal L $ is the linear operator on the left hand side of system \eqref{sys:evolutionary-op}. 

Taking the $ L^2 $-inner product of \eqref{sys:evl-op-shorthand} with $ 2 e^{-\varepsilon g(z)}\Lambda $ and applying integration by parts in the resultant lead to 
\begin{equation}
    \label{ene-001}
\begin{gathered}
    \dfrac{d}{dt} \norm{\frac{\xi}{e^{\varepsilon g/2}}, \frac{V}{e^{\varepsilon g/2}}, \frac{\varepsilon W}{e^{\varepsilon g/2}}}{L^2}^2 + {2}
    \int \frac{\vert \nabla V \vert^2 + \vert \varepsilon \nabla W \vert^2}{\rho e^{\varepsilon g(z)}} \,d\vec x \\
    = \int \bigl\lbrack \dvh \bigl( e^{-\varepsilon g(z)} \vec v \bigr) + \dz \bigl(e^{-\varepsilon g(z)} w \bigr) \bigr\rbrack \bigl( |\xi|^2+\vert V \vert^2 + \varepsilon^2 \vert W \vert^2 \bigr) \,d\vec x \\
    - 2 \int V^\top\nabla V \nabla (\frac{1}{\rho e^{\varepsilon g}}) \,d\vec x - 2 {\varepsilon^2} \int W \nabla W \cdot \nabla(\frac{1}{\rho e^{\varepsilon g(z)}}) \,d\vec x\\
    + 2 \int \bigl( G_1 \xi e^{-\varepsilon g(z)} + G_2 \cdot V e^{-\varepsilon g(z)} + \varepsilon G_3 W e^{-\varepsilon g(z)} \bigr) \,d\vec x =: \sum_{j=1}^{4} J_j.
\end{gathered} 
\end{equation}
\begin{remark}
\label{rm:weights}
    The weight $ e^{-\varepsilon g} $ in the energy estimate \eqref{ene-001} is to incorporate the linear term $ W \dz g $ in \eqref{eq:evl-op-1}, such that all terms on the right-hand side of \eqref{ene-001} are at least tri-linear and do not lose any regularity. See Remark \ref{rm:loss-regularity}, above. 
    
\end{remark}

First, for $ \varepsilon \in (0,1) $ small enough, one has that
\begin{equation}
    \label{ene-002}
\begin{aligned}
    \mathcal E_0(\xi,V,W) & :=\norm{\frac{\xi}{e^{\varepsilon g/2}}, \frac{V}{e^{\varepsilon g/2}}, \frac{\varepsilon W}{e^{\varepsilon g/2}}}{L^2}^2 \geq \mathfrak c_0 \norm{\xi, V, \varepsilon W}{L^2}^2,\\
    \mathcal D_0(V, W) & :=
    \int \frac{\vert \nabla V \vert^2 + \vert \varepsilon \nabla W \vert^2}{\rho e^{\varepsilon g(z)}} \,d\vec x  
    \geq \mathfrak c_0 \norm{\nabla V,  \varepsilon \nabla W}{L^2}^2,
\end{aligned}
\end{equation}
for some constants $ \mathfrak c_0 \in (0,\infty) $.

Next, we will estimate $ J_j $, $ j = 1,2,3,4 $, on the right hand side of \eqref{ene-001}. Recalling \eqref{def:symmetric_variable}, one has that
\begin{equation}
    \label{ene-004}
    \begin{aligned}
        J_1 \lesssim (\norm{\vec v, w,\nablah \vec v, \dz w}{L^\infty}^2 + 1) \mathcal E_0(\xi,V,W).
    \end{aligned}
\end{equation}
Hereafter, to shorten the notation, we use $ A \lesssim B $ to represent $ A \leq CB $ for some $ C $ independent of $ \varepsilon $. 
Meanwhile, one can verify that
\begin{equation}
    \label{ene-006}
    \begin{gathered}
    J_2 + J_3 \lesssim (\norm{\nabla q}{L^\infty} + 1)\norm{V,\varepsilon W}{L^2} \norm{\nabla V, \varepsilon \nabla W}{L^2}.
\end{gathered}
\end{equation}
Therefore, applying the Cauchy-Schwarz inequality in \eqref{ene-006} yields 
\begin{equation}
    \label{ene-007}
    \begin{gathered}
    J_2 + J_3 \leq \frac{1}{2} \mathcal D_0(V, W) 
     + C (1 + \norm{\nabla q}{L^\infty}^2)\mathcal E_0(\xi,V,W),
\end{gathered}
\end{equation}
for some constant $ C \in (0,\infty) $. 
Directly, we have
\begin{equation}
    J_4 \lesssim \|G_1, G_2, G_3\|_{L^2}  \mathcal E_0(\xi,V,W)^{\frac12}.
\end{equation}

We have proved the following lemma:
\begin{lemma}
    \label{lm:energy-structure}
    Under the {\it a priori} assumption \eqref{ene:a-priori-rho}, for $ \varepsilon $ small enough,  any solution $ (\xi, V, W) $ to \eqref{sys:evl-op-shorthand} satisfies 
    \begin{equation}
    \label{ene-009}
    \begin{gathered}
    \dfrac{d}{dt} \mathcal E_0(\xi, V,W) + \frac{1}{2}\mathcal D_0(V, W) \lesssim (\norm{ \vec v, w, \nablah \vec v, \dz w,  \nabla q}{L^\infty}^2 + 1) \mathcal E_0(\xi, V, W) \\
    + \norm{G_1, G_2, G_3}{L^2}\mathcal E_0(\xi, V, W)^{1/2}.
    \end{gathered}
    \end{equation}    
\end{lemma}

\subsection{Acoustic wave structure}
\label{subsec:ln-acoustic-wave}

The energy evolutionary obtained in Lemma \ref{lm:energy-structure} is not enough to close the energy estimate for the nonlinear problem. In particular, the estimate of $ W $ in \eqref{ene-009} becomes singular as $ \varepsilon \rightarrow 0 $. To overcome this degeneracy/singularity, one needs to establish the evolution of the acoustic wave energy. 

\smallskip 
To be more precise, 
let $ (\xi_\varepsilon,V_\varepsilon,W_\varepsilon):= (\xi e^{-\varepsilon g}, V e^{-\varepsilon g}, W e^{-\varepsilon g}) $.
Multiplying \eqref{sys:evolutionary-op} with $ e^{-\varepsilon g} $ leads to 
\begin{gather}
    \label{apndx-eq:001}
\begin{gathered}
     \dt \xi_\varepsilon +  \vec v \cdot \nablah \xi_\varepsilon +  w \dz \xi_\varepsilon + \frac{1}{\varepsilon}(\dvh V_\varepsilon + \dz W_\varepsilon) 
    = G_1 e^{-\varepsilon g} -  (\varepsilon w) \xi_\varepsilon \dz g,
\end{gathered}\\
\label{apndx-eq:002}
\begin{gathered}
    \dt V_\varepsilon + \vec v \cdot \nablah V_\varepsilon  + w \dz V_\varepsilon + \frac{1}{\varepsilon} \nablah \xi_\varepsilon - \frac{1}{\rho}\Delta V_\varepsilon 
    = {G_2e^{-\varepsilon g}}
    -(\varepsilon w) V_\varepsilon \dz g \\
    + \frac{1}{\rho} \bigl\lbrack \dz(\varepsilon \dz g V_\varepsilon) + \varepsilon \dz  g \dz V_\varepsilon + (\varepsilon \dz g)^2 V_\varepsilon \bigr\rbrack,
\end{gathered}\\
\label{apndx-eq:003}
\begin{gathered}
    \varepsilon^2 (\dt W_\varepsilon + \vec v \cdot \nablah W_\varepsilon + w \dz W_\varepsilon) + \frac{e^{-\varepsilon g}}{\varepsilon} \dz \xi - \frac{\varepsilon^2}{\rho} \Delta W_\varepsilon
    = {\varepsilon G_3e^{-\varepsilon g}} \\
    - \varepsilon^3 w W_\varepsilon \dz g 
    + \frac{\varepsilon^2}{\rho} \bigl\lbrack \dz (\varepsilon \dz g W_\varepsilon) + \varepsilon \dz g \dz W_\varepsilon + (\varepsilon \dz g)^2 W_\varepsilon \bigr\rbrack.
\end{gathered}
\end{gather}

Next, applying $ \dt $ to \eqref{apndx-eq:001} yields
\begin{equation}
    \label{apndx-eq:004}
    \begin{gathered}
         \partial_{tt} \xi_\varepsilon  + \vec v \cdot \nablah \dt \xi_\varepsilon +  w \dz \dt \xi_\varepsilon + \frac{1}{\varepsilon}(\dvh \dt V_\varepsilon + \dz \dt W_\varepsilon ) \\ 
        = - \dt\vec{v} \cdot \nablah \xi_\varepsilon - \dt w \dz \xi_\varepsilon - \dt (\varepsilon w \xi_\varepsilon \dz g) 
         + \dt (G_1 e^{-\varepsilon g}) =: \frac{H_1}{\varepsilon}.
    \end{gathered}
\end{equation}
From \eqref{apndx-eq:002} and \eqref{apndx-eq:003}, one can calculate that
\begin{equation}
    \label{apndx-eq-006}
    \begin{gathered}
    \dvh \dt V_\varepsilon + \dz \dt W_\varepsilon = 
    - \frac{1}{\varepsilon}\Delta_h \xi_\varepsilon - \frac{1}{\varepsilon^3} \dz ( e^{-\varepsilon g} \partial_{z} \xi) 
    + \frac{1}{\rho} \Delta (\dvh V_\varepsilon + \dz W_\varepsilon)\\
    - \dvh (\vec v \cdot \nablah V_\varepsilon + w \dz V_\varepsilon) - \dz (\vec v \cdot \nablah W_\varepsilon + w \dz W_\varepsilon) 
    + \nablah (\frac{1}{\rho}) \cdot \Delta V_\varepsilon
    + \dz (\frac{1}{\rho}) \Delta W_\varepsilon \\
    + \dvh \biggl\lbrace {G_2 e^{-\varepsilon g}-\varepsilon w V_\varepsilon \dz g}  + \frac{1}{\rho} \biggl\lbrack \substack{\dz (\varepsilon \dz g V_\varepsilon) + \varepsilon \dz g \dz V_\varepsilon \\
    +(\varepsilon \dz g)^2 V_\varepsilon} \biggr\rbrack \biggr\rbrace \\
    + \dz \biggl\lbrace {\frac1\varepsilon G_3 e^{-\varepsilon g}}- \varepsilon w W_\varepsilon \dz g  + \frac{1}{\rho} \biggl\lbrack \substack{ \dz (\varepsilon \dz g W_\varepsilon) + \varepsilon \dz g \dz W_\varepsilon \\
    +(\varepsilon \dz g)^2 W_\varepsilon
    }\biggr\rbrack \biggr\rbrace\\
    =: - \frac{1}{\varepsilon}\Delta_h \xi_\varepsilon - \frac{1}{\varepsilon^3} \dz ( e^{-\varepsilon g} \partial_{z} \xi) 
    + \frac{1}{\rho} \Delta (\dvh V_\varepsilon + \dz W_\varepsilon)  - H_2,
\end{gathered}
\end{equation}
while, from \eqref{apndx-eq:001}, one has that 
\begin{equation}
    \label{apndx-eq:007}
    \begin{aligned}
        \Delta (\dvh V_\varepsilon + \dz W_\varepsilon) = &- \varepsilon \dt \Delta  \xi_\varepsilon - \varepsilon \vec v \cdot \nablah \Delta \xi_\varepsilon - \varepsilon w \dz \Delta \xi_\varepsilon 
        - \varepsilon \Delta\vec v\cdot \nablah \xi_\varepsilon 
        - 2 \varepsilon \nabla\vec v: \nablah \nabla \xi_\varepsilon  \\
        &- \varepsilon \Delta w \dz \xi_\varepsilon - 2 \varepsilon \nabla w \cdot\dz \nabla \xi_\varepsilon  
        + \varepsilon \Delta (G_1 e^{-\varepsilon g} - (\varepsilon w) \xi_\varepsilon \dz g ) \\
        =: &- \varepsilon \dt \Delta  \xi_\varepsilon - \varepsilon \vec v \cdot \nablah \Delta \xi_\varepsilon - \varepsilon w \dz \Delta \xi_\varepsilon - H_3.
    \end{aligned}
\end{equation}
Combining \eqref{apndx-eq:004}--\eqref{apndx-eq:007} yields
\begin{equation*}
    \begin{gathered}
        \partial_{tt} \xi_\varepsilon - \frac{1}{\rho} \dt \Delta \xi_\varepsilon + \vec v \cdot \nablah \dt \xi_\varepsilon +  w \dz \dt \xi_\varepsilon 
        - \frac{1}{\rho} (\vec v \cdot \nablah \Delta \xi_\varepsilon + w \dz \Delta \xi_\varepsilon) \\
        - \frac{1}{\varepsilon^2} \Delta_h \xi_\varepsilon - \frac{1}{\varepsilon^4} \dz(e^{-\varepsilon g}\dz \xi)  = \frac{H_1}{\varepsilon} + \frac{H_2}{\varepsilon} + \frac{H_3}{\varepsilon \rho} 
    \end{gathered}
\end{equation*}
or, equivalently,
\begin{equation}
    \label{eq:linear-aw}
    \begin{aligned}
        &\dt(\dt\xi_\varepsilon -  \frac{1}{\rho}\Delta \xi_\varepsilon ) +  \vec v\cdot \nablah (\dt\xi_\varepsilon - \frac{1}{\rho}\Delta \xi_\varepsilon ) +  w \dz (\dt \xi_\varepsilon - \frac{1}{\rho} \Delta \xi_\varepsilon) 
        - \frac{1}{\varepsilon^2} \Delta_h \xi_\varepsilon - \frac{1}{\varepsilon^4} \dz (e^{-\varepsilon g} \dz \xi)\\ 
        = &\frac{H_1}{\varepsilon} + \frac{H_2}{\varepsilon} + \frac{H_3}{\varepsilon\rho} + \frac{\dt \rho}{\rho^2} \Delta \xi_\varepsilon  + \frac{1}{\rho^2} \vec v \cdot \nablah \rho \Delta \xi_\varepsilon + \frac{1}{\rho^2} w \dz \rho \Delta \xi_\varepsilon 
        \\
        =: &\frac{H_1}{\varepsilon} + \frac{H_2}{\varepsilon} + \frac{H_3}{\varepsilon\rho} + \frac{H_4}{\varepsilon}.
    \end{aligned}
\end{equation}

Now we are ready to establish the acoustic wave energy estimate. Taking the $ L^2 $-inner product of \eqref{eq:linear-aw} with $ \varepsilon^2 (\dt \xi - \frac{e^{\varepsilon g}}{\rho} \Delta (\xi e^{-\varepsilon g}) ) $, after applying integration by parts, leads to
\begin{equation}
    \label{ac-ene-001}
    \begin{gathered}
    \frac{1}{2}\dfrac{d}{dt} \norm{\frac{\varepsilon \lbrack \dt \xi - \frac{e^{\varepsilon g}}{\rho} \Delta (\xi e^{-\varepsilon g})\rbrack}{ e^{\varepsilon g/2}}, \frac{\nablah \xi}{e^{\varepsilon g/2}}, \frac{\dz \xi}{\varepsilon e^{\varepsilon g/2}}}{L^2}^2 
    + \norm{\frac{\Delta_h \xi}{\rho^{1/2}e^{\varepsilon g/2}}}{L^2}^2 + \norm{e^{\varepsilon g/2}\frac{\partial_{zz}(\xi e^{-\varepsilon g})}{\varepsilon \rho^{1/2}}}{L^2}^2 \\ + (\varepsilon^2 + 1)\norm{\frac{\nablah \dz \xi}{\varepsilon \rho^{1/2}e^{\varepsilon g/2}}}{L^2}^2 
    = \frac{1}{2} \int \biggl\lbrack \substack{ \dvh(\frac{\vec v}{e^{\varepsilon g}}) +\dz(\frac{w}{ e^{\varepsilon g}}) \\ + 2 \frac{\varepsilon w \dz g }{e^{\varepsilon g}} } \biggr\rbrack \vert \varepsilon\lbrack \dt \xi  - \frac{e^{\varepsilon g}}{\rho} \Delta (\xi e^{-\varepsilon g})\rbrack\vert^2 \,d\vec x \\
    + \int \biggl\lbrace \substack{ {-}\frac{{e^{\varepsilon g}}\dz(\xi e^{-\varepsilon g}\dz g) \partial_{zz} (\xi e^{-\varepsilon g})}{\varepsilon \rho} + \frac{\nablah \rho \cdot \nablah \dz \xi \dz \xi}{\varepsilon^2 \rho^2 e^{\varepsilon g}} \\
     + \frac{\nablah \rho \cdot \nablah \dz \xi \partial_{z} (\xi e^{-\varepsilon g})}{\rho^2}
    + \frac{\varepsilon \dz g \nablah \xi  \cdot \nablah \dz \xi}{\rho e^{\varepsilon g}}  \\
     + \dz(\frac{1}{\rho}) \Delta_h \xi  \dz (\xi e^{-\varepsilon g}) + \frac{1}{\varepsilon^2} \dz(\frac{1}{\rho})\Delta_h \xi \dz \xi e^{-\varepsilon g} } \biggr\rbrace  \,d\vec x \\
     + \int (H_1 + H_2 + \frac{H_3}{\rho} + H_4) \lbrack \varepsilon (\dt \xi  - \frac{e^{\varepsilon g}}{\rho}\Delta (\xi e^{-\varepsilon g})) \rbrack \,d \vec x 
     =: \sum_{j = 5}^{7} J_j.
    \end{gathered}
\end{equation}
Meanwhile, taking the $ L^2 $-inner product of \eqref{eq:linear-aw} with $ \varepsilon^2 \dt \xi $ yields that 
\begin{equation}
    \label{ac-ene-002}
    \begin{gathered}
        \frac{1}{2}\dfrac{d}{dt}\norm{\frac{\varepsilon \dt \xi}{e^{\varepsilon g/2}}, \dfrac{\nablah \xi}{e^{\varepsilon g/2}}, \frac{\dz\xi}{\varepsilon e^{\varepsilon g/2}}}{L^2}^2 
        + \norm{\frac{\varepsilon \nabla \dt \xi}{\rho^{1/2} e^{\varepsilon g/2}}}{L^2}^2 
        = \frac{1}{2} \int \lbrack \dvh (\frac{\vec v}{e^{\varepsilon g}}) + \dz(\frac{w}{ e^{\varepsilon g}}) + 2 \frac{\varepsilon w \dz g}{e^{\varepsilon g}} \rbrack  \vert \varepsilon \dt \xi  \vert^2 \,d\vec x \\
        - \int \varepsilon^2 ( \dvh \vec v + \dz w ) \frac{1}{\rho}\Delta(\xi e^{-\varepsilon g})\dt \xi  \,d\vec x 
        - \int \varepsilon^2 \bigl( \frac{\vec v}{\rho}\cdot\nablah \dt \xi + \frac{w}{\rho} \dz \dt \xi \bigr) \Delta(\xi e^{-\varepsilon g}) \,d\vec x \\
        + \int \biggl( \substack{ \dt(\frac{\varepsilon^2}{\rho}) \Delta(\xi e^{-\varepsilon g}) \dt \xi  - \nabla (\frac{\varepsilon^2}{\rho}) \cdot \nabla \dt(\xi e^{-\varepsilon g}) \dt\xi \\
        + \frac{\varepsilon^3\dz g  \dt \xi \dz \dt \xi}{\rho e^{\varepsilon g}} } \biggr) \,d\vec x 
        + \int (H_1 + H_2 + \frac{H_3}{\rho} + H_4) \varepsilon \dt \xi \,d\vec x =: \sum_{j=8}^{12}J_j.
    \end{gathered}
\end{equation}

For $ \varepsilon \in (0,1) $ small enough, one has that
\begin{equation}
    \label{ac-ene-003}
    \begin{aligned}
        \mathcal E_1(\xi) &:= \norm{\frac{\varepsilon \lbrack \dt \xi - \frac{e^{\varepsilon g}}{\rho}\Delta (\xi e^{-\varepsilon g})\rbrack}{ e^{\varepsilon g/2}}, \frac{\varepsilon \dt \xi}{e^{\varepsilon g/2}}}{L^2}^2 + 2\norm{\frac{\nablah \xi}{e^{\varepsilon g/2}}, \frac{\dz \xi}{\varepsilon e^{\varepsilon g/2}}}{L^2}^2\\
        & \qquad \geq \mathfrak c_1 \norm{\varepsilon \dt \xi, \varepsilon \nabla^2 \xi, \nablah \xi, \frac{\dz \xi}{\varepsilon}}{L^2}^2 - \|\xi\|_{L^2}^2, \\
        \mathcal D_1(\xi) &:= \norm{\frac{\Delta_h \xi}{\rho^{1/2} e^{\varepsilon g/2}},e^{\varepsilon g/2}\frac{\partial_{zz}(\xi e^{-\varepsilon g})}{\varepsilon \rho^{1/2}}, \frac{\nablah \dz \xi}{\varepsilon \rho^{1/2} e^{\varepsilon g/2}},\frac{\varepsilon \nabla \dt \xi}{\rho^{1/2} e^{\varepsilon g/2}}}{L^2}^2 \\
        & \qquad \geq \mathfrak c_1 \norm{\Delta \xi, \frac{\nabla \dz \xi}{\varepsilon},\varepsilon \nabla \dt \xi}{L^2}^2- \|\xi\|_{L^2}^2,
    \end{aligned}
\end{equation}
where the subtraction of $ \|\xi\|_{L^2}^2$ comes from when the $z$ derivatives hit on $e^{-\varepsilon g}$ in $\Delta(\xi e^{-\varepsilon g})$ and $\partial_{zz} (\xi e^{-\varepsilon g})$.
Now we are ready to estimate $ J_j $, $ j = 5,6, \cdots, 12 $. 
Applying H\"older's inequality and the Sobolev embedding inequality yields that
\begin{equation}
    \label{ac-ene-004}
    \begin{aligned}
        J_5 + J_{8} + J_{9} & \lesssim (\norm{
            \vec v, \nablah \vec v, w, \dz w}{L^\infty}^2 + 1) \mathcal E_1,\\
        J_6 + J_{10} + J_{11} & \lesssim (\norm{\varepsilon \dt q,\nabla q,\vec v, w}{L^\infty} + 1)(\mathcal E_1^{1/2}+ \|\xi\|_{L^2}) (\mathcal D_1^{1/2} + \|\xi\|_{L^2}).
    \end{aligned}
\end{equation}
Meanwhile, one has that
\begin{equation}
    \label{ac-ene-005}
        J_7 + J_{12} \lesssim \norm{H_1,H_2,H_3,H_4}{L^2} \mathcal E_1^{1/2}.
\end{equation}
It remains to estimate the $ H_j $'s. From \eqref{apndx-eq:004}, one can calculate that
\begin{equation}
    \label{ac-ene-006}
    \begin{aligned}
        \norm{H_1}{L^2} & \lesssim (\norm{\varepsilon \dt \vec v, \varepsilon^2 \dt w,\varepsilon w}{L^\infty} + 1) (\mathcal E_1^{1/2}+ \|\xi\|_{L^2})  + \norm{\varepsilon \dt G_1}{L^2}.
    \end{aligned}
\end{equation}
Similarly, from \eqref{apndx-eq-006}--\eqref{eq:linear-aw}, one can calculate that
\begin{align}
    \label{ac-ene-007}
    & \begin{aligned}
        \norm{H_2}{L^2} & \lesssim (\norm{\vec v,{\nabla} \vec v, w, \dz w, \varepsilon \nabla_h w, \nabla q}{L^\infty} +1) \norm{\nabla^2 V,  V, \nabla W, \nabla \dz W, \varepsilon \nabla^2 W}{L^2} \\
        & \qquad + \norm{\nabla w}{L^\infty} \norm{V,\nabla V}{L^2}
        + \norm{\nablah G_2, \frac{\dz G_3}{\varepsilon},G_3}{L^2},
    \end{aligned}\\
    \label{ac-ene-008}
    & 
    \begin{aligned}
        \norm{H_3}{L^2} & \lesssim \norm{ \nabla \vec v, \vec v,\varepsilon \nabla w, \varepsilon w}{L^\infty} (\mathcal E_1^{1/2}+ \mathcal D_1^{1/2}  + \|\xi\|_{L^2}) 
        \\
        &\qquad  + \|\varepsilon \nabla^2 w,\nabla^2 \vec v\|_{L^\infty} (\mathcal E_1^{1/2}+ \|\xi\|_{L^2}) + \norm{\varepsilon \nabla^2 G_1, \varepsilon G_1}{L^2},
    \end{aligned}
    \\
\intertext{and}
    \label{ac-ene-009}
    & \norm{H_4}{L^2} \lesssim (\norm{\vec v, w, \nabla q,\varepsilon \dt q}{L^\infty}^2 + 1) (\mathcal E_1^{1/2} + \|\xi\|_{L^2}).
\end{align}

Collecting \eqref{ac-ene-001}--\eqref{ac-ene-009} and applying Cauchy-Schwartz's inequality yield the following lemma 
\begin{lemma}
    \label{lm:acoustic-energy-flow}
    Under the {\it a priori} assumption \eqref{ene:a-priori-rho}, for small enough $ \varepsilon \in (0,1) $, any solution $ (\xi, V, W) $ to \eqref{sys:evl-op-shorthand} satisfies
\begin{equation}
    \label{ac-ene-010}
    \begin{gathered}
        \dfrac{d}{dt} \mathcal E_1(\xi) + \mathcal D_1(\xi) \lesssim \mathcal H(\norm{\varepsilon \dt q, \nabla q,  \vec v, \nabla \vec v, w, \dz w,\varepsilon \nabla w}{L^\infty})(\mathcal E_1 + \|\xi\|_{L^2}^2)\\
        + \norm{\varepsilon \dt \vec v, \varepsilon^2 \dt w}{L^\infty} (\mathcal E_1 + \|\xi\|_{L^2}^2)\\
        + \mathcal H(\norm{\vec v, \nabla \vec v, w,\dz w, \varepsilon \nabla w,\nabla q}{L^\infty})\norm{V,\nabla^2 V, \nabla W, \varepsilon \nabla^2 W, \nabla \dz W}{L^2}\mathcal E_1^{1/2}\\
        + \norm{\nabla w}{L^\infty}\norm{V,\nabla V}{L^2}\mathcal E_1^{1/2} +\|\varepsilon \nabla^2 w,\nabla^2 \vec v\|_{L^\infty} (\mathcal E_1+ \|\xi\|^2_{L^2})\\
        + \norm{\varepsilon \dt G_1,\varepsilon \nabla^2 G_1,\varepsilon G_1,\nablah G_2, \frac{\dz G_3}{\varepsilon},G_3}{L^2} \mathcal E_1^{1/2}.
    \end{gathered}
\end{equation}
\end{lemma}

\section{Uniform-in-$ \varepsilon $ estimate}
\label{sec:uniform-est}

\subsection{Energy functional and elliptic estimate}
\label{subsec:ell-est}

Let the intermediate energy and dissipation functionals be 
\begin{align}
    \label{def:itm-eng}
    \mathfrak E(t) &:= \norm{q,\vec v, \varepsilon w}{H^3}^2 + \norm{\varepsilon \dt q,\varepsilon \nabla^2 q,\frac{\dz q}{\varepsilon}}{H^2}^2,\\
    \label{def:itm-dis}
    \mathfrak D(t) & := \norm{\nabla \vec v, \varepsilon \nabla w}{H^3}^2 + \norm{\Delta q, \frac{\nabla \dz q}{\varepsilon}, \varepsilon \nabla \dt q}{H^2}^2.
\end{align}
Notice that we have the following equivalence between intermediate energy and dissipation with energy/acoustic energy and dissipation/acoustic dissipation:
\begin{lemma}\label{lm:equivalence}
    There exists a constant $\mathfrak c_3>1$ such that, for $ \varepsilon \in (0,1) $ small enough, 
    \begin{gather}
    \label{ene-fncts-eqvlt}
    \begin{aligned}
        \frac{1}{\mathfrak c_3}\mathcal E_\mrm{total}\leq \mathfrak E 
        \leq \mathfrak c_3 \mathcal E_\mrm{total},
    \end{aligned} \\
    \intertext{and}
    \label{dss-fncts-eqvlt}
    \begin{aligned}
        \frac{1}{\mathfrak c_3}\bigl(\mathcal D_\mrm{total} - \mathfrak E\bigr)  \leq \mathfrak D 
        \leq \mathfrak c_3 \bigl(\mathcal D_\mrm{total}+ \mathfrak E\bigr),
    \end{aligned}  
    \end{gather}
    where 
    \begin{align}
        \mathcal E_\mrm{total}:= &\sum_{j=0}^3\mathcal E_0(\nabla^jq,\nabla^jv,\nabla^jw) + \sum_{j=0}^2\mathcal E_1(\nabla^j q), \label{def:total-energy}
        \\
        \mathcal D_\mrm{total}:=&\sum_{j=0}^3\mathcal D_0(\nabla^jv,\nabla^jw) + \sum_{j=0}^2\mathcal D_1(\nabla^j q). \label{def:total-dissipation}
    \end{align}
    Here $ \mathcal E_0, \mathcal D_0, \mathcal E_1, \mathcal D_1 $ are given in \eqref{ene-002} and \eqref{ac-ene-003}.
\end{lemma}

Then we have the following elliptic estimates:
\begin{proposition}\label{prop:elliptic-est}
For $ \varepsilon \in (0,1) $ small enough, one has that
    \begin{gather}
        \label{ellest:w}
        \norm{w, \dz w}{H^2} \lesssim \mathcal H(\mathfrak E(t)), \\
        \label{ellest:w-d}
        \norm{w,\dz w}{H^3} \lesssim \mathcal H(\mathfrak E(t))\mathfrak D(t)^{1/2} + \mathcal H(\mathfrak E(t))
        \\
        \label{ellest:dt-v-w}
        \norm{\varepsilon^2 \dt w,\varepsilon \dt \vec v}{H^2}\lesssim \mathcal H(\mathfrak E(t)) \mathfrak D(t)^{1/2} + \mathcal H(\mathfrak E(t)).
    \end{gather}
\end{proposition}

\begin{proof}
    From \eqref{eq:w}, one can directly apply Minkowski's inequality, H\"older's inequality, and the Sobolev embedding inequality to obtain
    \begin{equation}
        \label{ellest:001}
        \begin{gathered}
        \norm{w, \dz w}{L^2} \lesssim (1+ \norm{\vec v}{L^\infty}^2 + \|\varepsilon\partial_z q\|_{L^\infty}^2)\norm{\varepsilon \dt q, \varepsilon  \nablah q, \dvh \vec v}{L^2}  \\
        \lesssim (1 + \norm{\vec v,  q}{H^3}^2) \norm{\varepsilon \dt q, \varepsilon  \nablah q, \dvh \vec v}{L^2} \lesssim \mathfrak E(t)^{3/2} + 1.
    \end{gathered}
    \end{equation}
    Similarly, with tedious but straightforward calculation, one can obtain \eqref{ellest:w} and \eqref{ellest:w-d}. 
    
    \smallskip 

    On the other hand, from \eqref{hbsym-h-momentum} and \eqref{hbsym-v-momentum}, one can write down 
    \begin{gather}
        \label{ellest:002}
        \varepsilon \dt \vec v = \frac{\varepsilon}{\rho} \Delta \vec v - \nablah q - \varepsilon(\vec v \cdot \nablah \vec v + w \dz \vec v) , \\
        \label{ellest:003}
        \varepsilon^2 \dt w  
    = \frac{\varepsilon^2}{\rho} \Delta w 
    - \varepsilon^2  (\vec v \cdot \nablah w + w \dz w) - \frac{\dz q}{\varepsilon}. 
    \end{gather}
    \eqref{ellest:dt-v-w} follows directly after applying H\"older's inequality and the Sobolev embedding inequality. 
\end{proof}

\subsection{Energy and acoustic wave estimates}
\label{subsec:ene-est}

We start by writing down the equations of the derivatives in the form of \eqref{sys:evl-op-shorthand}. In fact, one can directly calculate from \eqref{sys:hst-bsnq-sym} that, for $ k =0,1,2,3 $, 
\begin{align}
    \label{eq:ddd-u}
    \mathcal L(t, \rho, \vec v, w,\varepsilon)\begin{pmatrix}
        \partial^k q \\ \partial^k \vec v \\ \partial^k w
    \end{pmatrix} = \begin{pmatrix}
        G_{1,k} \\ G_{2,k} \\ \varepsilon G_{3,k}
    \end{pmatrix}
\end{align}
where $ \mathcal L $ is as in \eqref{sys:evl-op-shorthand}, $ \partial \in \lbrace \partial_x,\partial_y,\partial_z \rbrace $, and
\begin{align}
    \label{def:g-1-ddd}
    & \begin{aligned}
        G_{1,k} &:= \sum_{j=1}^{k} \biggl( \partial^{k-j} w \dz \partial^j g - \partial^j \vec v \cdot\nablah \partial^{k-j} q - \partial^j w \dz \partial^{k-j} q \biggr), 
    \end{aligned}\\
    & \label{def:g-2-ddd}
    \begin{aligned}
        G_{2,k} &:= - \sum_{j=1}^{k} \biggl( \partial^j \vec v \cdot \nablah \partial^{k-j} \vec v +  \partial^j w \dz \partial^{k-j} \vec v \biggr) 
        + \sum_{j=1}^{k} \partial^j (\frac{1}{\rho}) \partial^{k-j}\Delta \vec v,
    \end{aligned}\\
    & \label{def:g-3-ddd}
    \begin{aligned}
        G_{3,k} &:= - \varepsilon \sum_{j = 1}^{k} \biggl(\partial^j \vec v \cdot \nablah \partial^{k-j} w + \partial^j w \dz \partial^{k-j} w \biggr)
        + \sum_{j=1}^k \partial^j (\frac{\varepsilon}{\rho}) \partial^{k-j} \Delta w, 
    \end{aligned}
\end{align}
where we have omitted the combinatorial constants and the multi-indices are suppressed. We also have $ G_{1,0} =G_{2,0} = G_{3,0} = 0 $. 

\smallskip

Now we are ready to establish the uniform energy estimate:
\begin{proposition}[Energy estimate]
    \label{prop:uniform-est}
    We have, for $ k = 0, 1,2,3 $, 
    \begin{equation}
        \label{ene-est:g-s}
        \begin{gathered}
        \norm{G_{1,k},G_{2,k},G_{3,k}}{L^2} \leq \mathcal H(\mathfrak E(t))(1 + \mathfrak D(t)^{1/2}),
    \end{gathered}
    \end{equation}
    and hence 
    \begin{equation}
        \label{ene-est:ene-ddd}
        \begin{gathered}
        \dfrac{d}{dt} \sum_{k=0}^3 \mathcal E_0(\nabla^k q, \nabla^k \vec v, \nabla^k w) + \frac{1}{2} \sum_{k=0}^3 \mathcal D_0(\nabla^k \vec v, \nabla^k w) \\
        \leq \mathcal H(\mathfrak E(t))(1 + \mathfrak D(t)^{1/2}).
        \end{gathered}
    \end{equation}
\end{proposition}

\begin{proof}
    We only show the high order terms, while the low order terms can be handled similarly. Directly, one can calculate that, from \eqref{def:g-1-ddd}, \eqref{def:g-2-ddd}, and \eqref{def:g-3-ddd} with $ k = 3 $,
    \begin{gather*}
        \norm{G_{1,3},G_{2,3},G_{3,3}}{L^2} \lesssim \norm{\partial^2 w}{L^2} + \norm{\partial^3 \vec v}{L^2} \norm{\nablah q, \nablah \vec v,\varepsilon \nablah w}{L^\infty}\\
        + \norm{\partial \vec v}{L^\infty} \norm{\nablah \partial^2 q, \nablah \partial^2 \vec v,\varepsilon \nablah \partial^2 w}{L^2} + \norm{\partial w}{L^\infty}\norm{\dz \partial^2 q, \dz \partial^2 \vec v,\varepsilon \dz\partial^2 w}{L^2} \\
        + \norm{\partial^3 w}{L^2}\norm{\dz q, \dz \vec v,\varepsilon \dz w}{L^\infty} + \norm{\partial^3 q, \partial^3 g}{L^2} \norm{\Delta \vec v,\varepsilon \Delta w}{L^\infty} \\ + \norm{\partial q, \partial g}{L^\infty} \norm{\partial^2 \Delta \vec v,\varepsilon \partial^2 \Delta w}{L^2} + \mrm{l.o.t}
        \lesssim \mathcal H(\mathfrak E(t))(1 + \mathfrak D(t)^{1/2}),
    \end{gather*}
    where we have used the Sobolev embedding inequality and \eqref{ellest:w-d}.

    To obtain \eqref{ene-est:ene-ddd}, one only need to apply Lemma \ref{lm:energy-structure} with $ (\xi, V, W) = (\partial^k q, \partial^k \vec v, \partial^k w) $ and $ (G_{1},G_{2},G_{3}) = (G_{1,k},G_{2,k},G_{3,k}) $, $ k = 0,1,2,3 $, and applying the Sobolev embedding inequality.

\end{proof}

Meanwhile, we have the uniform acoustic wave estimate:
\begin{proposition}[Acoustic wave energy estimate]
    \label{prop:uniform-aw}We have, for $ k = 0,1,2$, 
    \begin{equation}
        \label{ene-est:g-s-aw}
        \norm{\varepsilon \dt G_{1,k},\varepsilon \nabla^2 G_{1,k},\varepsilon G_{1,k},\nabla G_{2,k},\frac{\dz G_{3,k}}{\varepsilon},G_{3,k}}{L^2} \lesssim H(\mathfrak E(t)) (\mathfrak D(t)^{1/2} + 1),
    \end{equation}
    and hence
    \begin{equation}
        \label{ene-est:ene-ddd-aw}
        \dfrac{d}{dt} \sum_{k=0}^{2}\mathcal E_1(\nabla^k q) + \sum_{k=0}^{2} \mathcal D_1(\nabla^k q) \leq H(\mathfrak E(t)) (\mathfrak D(t)^{1/2} + 1)
    \end{equation}
\end{proposition}

\begin{proof}
    We focus on the high order estimates, while the low order estimates will follow similarly. 

\medskip
    
    {\par\noindent\bf Estimate of $ \varepsilon \dt G_{1,2} $:} With some tedious but straightforward calculation, one can write that
    \begin{equation}
    \label{ac-ene-est:000}
    \begin{gathered}
        \varepsilon \dt G_{1,2} = 
        \varepsilon \biggl( \substack{\partial \dt w \dz \partial g + \dt w \dz \partial^2 g - \partial \dt \vec v \cdot \nablah \partial q \\ - \partial^2 \dt \vec v \cdot \nablah q - \partial \dt w \dz \partial q - \partial^2 \dt w \dz q 
        } \biggr)
        + \varepsilon \biggl( \substack{- \partial \vec v \cdot \nablah \partial \dt q - \partial^2 \vec v \cdot \nablah \dt q \\ - \partial w \dz \partial \dt q - \partial^2 w \dz \dt q 
        } \biggr) =: I_1 + I_2.
    \end{gathered}
\end{equation}
    To estimate $ I_1 $, using the assumption that $ \dz \partial g = \mathcal O (\varepsilon) $ and $ \dz \partial^2 g = \mathcal O(\varepsilon) $, one has 
    \begin{equation}
        \label{ac-ene-est:001}
        \begin{aligned}
            \norm{I_1}{L^2} & \lesssim (\norm{\nabla q,\frac{\dz q}{\varepsilon}}{L^\infty} + 1) \norm{\substack{\varepsilon^2 \partial^2 \dt w,\varepsilon^2 \partial \dt w, \\ \varepsilon^2 \dt w,  \varepsilon \partial^2 \dt \vec v}}{L^2}
             + \norm{\varepsilon^2 \partial \dt w,\varepsilon \partial \dt \vec v}{L^3} \norm{\nabla^2 q,\frac{\dz \partial q}{\varepsilon}}{L^6}\\
            & \lesssim \mathcal H(\mathfrak E(t)) (\mathfrak D(t)^{1/2} + 1),
        \end{aligned}
    \end{equation}
    where we have applied the Sobolev embedding inequality, \eqref{ellest:w}, and \eqref{ellest:dt-v-w}.

    \begin{equation}
        \label{ac-ene-est:002}
    \begin{aligned}
         \norm{I_2}{L^2} & \lesssim \|\varepsilon\partial_t q\|_{H^2} \|\partial v, \partial w\|_{L^\infty} + \|\partial^2 v, \partial^2 w\|_{L^3} \|\varepsilon\partial_t \nabla q\|_{L^6} 
         \\
         & \lesssim \mathcal H(\mathfrak E(t)) (\mathfrak D(t)^{1/2} + 1),
    \end{aligned}
    \end{equation}
where we have used \eqref{ellest:w-d}.

\medskip

{\par\noindent\bf Estimate of $ \varepsilon \nabla^2 G_{1,2} $:} Similarly, from \eqref{def:g-1-ddd}, one can calculate that
\begin{equation}
\label{ac-ene-est:003}
    \begin{gathered}
        \varepsilon \nabla^2 G_{1,2} = \varepsilon \biggl( \substack{\partial \nabla^2 w \dz \partial g + \nabla^2 w \dz \partial^2 g - \partial \nabla^2 \vec v \cdot \nablah \partial q \\ - \partial^2 \nabla^2 \vec v \cdot \nablah q - \partial \nabla^2 w \dz \partial q - \partial^2 \nabla^2 w \dz q 
        } \biggr) 
        + \varepsilon \biggl( \substack{- \partial \vec v \cdot \nablah \partial \nabla^2 q - \partial^2 \vec v \cdot \nablah \nabla^2 q \\ - \partial w \dz \partial \nabla^2 q - \partial^2 w \dz \nabla^2 q 
        } \biggr) + \mrm{l.o.t.} \\ =: I_3 + I_4 + \mrm{l.o.t.},
    \end{gathered}
\end{equation}
where we have only written down the highest order terms.
With similar and straightforward estimates, one has that 
\begin{align}
    & 
    \label{ac-ene-est:004} 
    \begin{aligned}
        \|I_3\|_{L^2} & \lesssim \norm{\varepsilon \nabla w ,\nabla \vec v}{H^3} (1 + \norm{q}{H^3}) \lesssim (1+ \mathfrak E^{1/2})\mathfrak D^{1/2},
    \end{aligned}\\
    \intertext{and}
    & \label{ac-ene-est:005} \begin{aligned}
        \norm{I_4}{L^2} & \lesssim \norm{\vec v,\varepsilon w}{H^3} \norm{\Delta q}{H^2} \lesssim\mathfrak E^{1/2}\mathfrak D^{1/2}.
    \end{aligned}
\end{align}
$\varepsilon G_{1,2}$ is lower order term, and can be bounded similarly.

\medskip

{\noindent \bf Estimate of $\nabla G_{2,2}$:}\ Direct calculation yields from \eqref{def:g-2-ddd}, 
\begin{equation}
\label{ac-ene-est:006}
    \begin{gathered}
        \nabla G_{2,2} = \biggl( \substack{\partial \partial \Delta \vec v  \partial (\frac{1}{\rho}) + \partial  \Delta \vec v  \partial^2 (\frac{1}{\rho}) - \partial \partial \vec v \cdot \nablah \partial \vec v \\ - \partial^2 \partial \vec v \cdot \nablah \vec v - \partial \partial w \dz \partial \vec v - \partial^2 \partial w \dz \vec v 
        } \biggr) 
        +  \biggl( \substack{\partial \Delta \vec v  \partial  \partial (\frac{1}{\rho}) +   \Delta \vec v  \partial\partial^2 (\frac{1}{\rho}) - \partial \vec v \cdot \nablah \partial  \partial \vec v \\ - \partial^2  \vec v \cdot \nablah \partial \vec v -  \partial w \dz \partial \partial \vec v - \partial^2  w \dz \partial \vec v 
        } \biggr) =: I_5 + I_6,
    \end{gathered}
\end{equation}
As before, one has that
\begin{align}
    & \label{ac-ene-est:007} 
    \begin{aligned}
        \norm{I_5+I_6}{L^2} & \lesssim \norm{\vec v}{H^3} \norm{\vec v}{H^3} + \norm{w}{H^3} \norm{\vec v}{H^3} + \norm{q}{H^3} \norm{\nabla \vec v}{H^3}  
        \\
        &\lesssim  H(\mathfrak E(t)) (\mathfrak D(t)^{1/2} + 1),
    \end{aligned}
    \end{align}
where we have applied \eqref{ellest:w-d}.

\medskip 
{\noindent\bf Estimate of $ \frac{\dz G_{3,2}}{\varepsilon } $:} From \eqref{def:g-3-ddd}, one has that
\begin{equation}
    \label{ac-ene-est:009}
    \begin{gathered}
    \frac{\dz G_{3,2}}{\varepsilon} = \biggl( \substack{ -\partial \dz\vec v \cdot \nablah \partial w - \partial^2 \dz \vec v \cdot \nablah w
    - \partial \dz w \dz \partial w \\
    - \partial^2 \dz w \dz w + \partial \dz (\frac{1}{\rho}) \partial \Delta w + \partial^2 \dz (\frac{1}{\rho}) \Delta w
    }
    \biggr) 
    + \biggl( \substack{ -\partial \vec v \cdot \nablah \partial \dz w - \partial^2 \vec v \cdot \nablah  \dz w
    - \partial w \dz \partial \dz w \\
    - \partial^2 w \dz \dz w + \partial (\frac{1}{\rho}) \partial \Delta \dz w + \partial^2 (\frac{1}{\rho}) \Delta \dz w
    }
    \biggr) =: I_7 + I_8.
    \end{gathered}
\end{equation}
Again, one has that 
\begin{equation}
    \label{ac-ene-est:010}
    \norm{I_7 + I_8}{L^2} \lesssim (\norm{\vec v}{H^3} + \norm{w,\dz w}{H^2} ) \norm{ w,\dz w}{H^3} + \norm{\varepsilon \nabla^2 q}{H^2} \norm{w, \dz w}{H^3}.
\end{equation}
$G_{3,2}$ is lower order term, and can be bounded similarly.

Combining all the estimates on $G_{i,k}$, $ i = 1,2,3,\ k = 0,1,2 $, we have
\begin{equation*}
    \norm{\varepsilon \dt G_{1,k},\varepsilon \nabla^2 G_{1,k},\varepsilon G_{1,k},\nabla G_{2,k},\frac{\dz G_{3,k}}{\varepsilon},G_{3,k}}{L^2} \lesssim H(\mathfrak E(t)) (\mathfrak D(t)^{1/2} + 1),
\end{equation*}
which is exactly \eqref{ene-est:g-s-aw}.

\medskip

{\noindent\bf Estimate of \eqref{ene-est:ene-ddd-aw}:} 
Finally, applying Lemma \ref{lm:acoustic-energy-flow} with $ (\xi, V, W) = (\nabla^j q, \nabla^j \vec v, \nabla^j w) $, $ j = 0,1,2 $, and embedding inequalities, together with \eqref{ene-est:g-s-aw} lead to \eqref{ene-est:ene-ddd-aw}.

\end{proof}

\begin{remark}[Isentropic flow]
\label{rm:isentropic-obstruction}
For isentropic flow, i.e., $ p(\rho) = \rho^\gamma $ for $ \gamma > 1 $, the current acoustic wave estimate in Proposition \ref{prop:uniform-aw} does not close with the available energy. 
For instance, consider, \(p(\rho)=\frac12\rho^2\). The natural
symmetric variable is then given by
\[
        \rho=1+\varepsilon(q-g(z)).
\]
Then the continuity equation may be written as
\[
        \partial_tq+\vec v\cdot\nabla_hq+w\partial_zq-w\partial_zg
        +
        \left(\frac1\varepsilon+q-g\right)
        (\operatorname{div}_h\vec v+\partial_z w)=0.
\]
After applying \(\partial^\alpha\) for any $ \alpha \geq 1 $, one obtains the lower-order source of the form, for $ 0 < \beta \leq \alpha $,
\[
        \partial^\beta(q-g)
        \bigl(
        \operatorname{div}_h\partial^{\alpha-\beta}\vec v
        +\partial_z\partial^{\alpha-\beta}w
        \bigr).
\]
The acoustic estimate requires controlling the time derivative of the
commuted continuity source.  Therefore this term generates contributions of
the type, 
\[
        \varepsilon\,\partial^\beta(q-g)
        \bigl(
        \operatorname{div}_h\partial_t\partial^{\alpha-\beta}\vec v
        +\partial_z\partial_t\partial^{\alpha-\beta}w
        \bigr),
\]
corresponding to $ \varepsilon \dt G_{1,2} $ in \eqref{ac-ene-est:000}. In particular, the term 
$
        \varepsilon\,\partial^\beta(q-g)\,
        \partial_z\partial_t\partial^{\alpha-\beta}w
$
can not be controlled uniformly in $ \varepsilon $.
In fact, similar to the linear analysis \eqref{time-derivative-w}, one can only expect 
$
        \mathcal O(\varepsilon\partial_t w)
        \geq \mathcal O(\varepsilon^{-1})
$.
Hence the above source term is singular as $ \varepsilon \rightarrow 0 $, i.e., 
$$
        \varepsilon\,\partial^\beta(q-g)\,
        \partial_z\partial_t\partial^{\alpha-\beta}w \geq \mathcal O(\varepsilon^{-1}).
$$
\end{remark}

\subsection{Uniform-in-$\varepsilon$ estimate and proof of Theorem \ref{thm:uniform-est}}
\label{subsec:uniform-est}

Collecting Lemma \ref{lm:equivalence} and Propositions \ref{prop:uniform-est}--\ref{prop:uniform-aw}, we have shown
\begin{equation}
    \label{ene-ineq-total}
    \dfrac{d}{dt}\mathcal E_\mrm{total} +  \frac12 \mathcal D_{\mrm{total}} \leq \mathcal H(\mathcal E_\mrm{total}) (1+ \mathcal D_\mrm{total}^{1/2}), 
\end{equation}
where $ \mathcal E_\mrm{total} $ and $ \mathcal D_\mrm{total} $ are given in \eqref{def:total-energy} and \eqref{def:total-dissipation}, respectively. 
This leads to
\begin{equation}
    \label{ene-total}
    \sup_{0\leq t \leq T} \mathcal E_\mrm{total}(t) + \int_0^T \mathcal D_\mrm{total}(s)\,ds \leq C,
\end{equation}
for some $ T, C\in(0,\infty) $, independent of $ \varepsilon $ but may depend on $ \mathcal E_0 $. Moreover, it is easy to verify that \eqref{ene:a-priori-rho} holds for $ \varepsilon $ small enough thanks to \eqref{def:symmetric_variable}. Together with Lemma \ref{lm:equivalence} and Proposition \ref{prop:elliptic-est},
this concludes the proof of Theorem \ref{thm:uniform-est}.

\section{Asymptotic limit}
\label{sec:limit}

\subsection{Wave decomposition and compactness}
\label{subsec:wave-dcmpt}

It follows from Theorem \ref{thm:uniform-est}, that 
\begin{equation}
    \label{uniform-est-total-01}
    \norm{q,\vec v,\varepsilon w}{L^\infty(0,T;H^3)\cap L^2(0,T;H^4)} + \norm{w,\dz w,\frac{\dz q}{\varepsilon},\varepsilon \dt q}{L^\infty(0,T;H^2)\cap L^2(0,T;H^3)} < C < \infty,
\end{equation}
for some constant $ C \in (0,\infty) $. 

Meanwhile, recalling Definition \eqref{def:wave-proj}, we write 
\begin{equation}
    \label{wdpt-001}
    \mathcal U:=\begin{pmatrix}
        q \\ \vec v \\ w
    \end{pmatrix} = \mathcal U_\sigma + \mathcal U_\mrm{ha} + \mathcal U_\mrm{va},
\end{equation}
where
\begin{align}
    & \label{wdpt-002}
    \begin{aligned}
    \mathcal U_\mrm{ha} & := \mathcal P_\mrm{ha} (q,\vec v,w) = \begin{pmatrix}
        \overline q - m_q \\ U_\mrm{ha}
    \end{pmatrix} \\& \quad \text{with} \ U_\mrm{ha} = \begin{pmatrix}
        \vec v_\mrm{ha} \\ w_\mrm{ha}=0
    \end{pmatrix} := P_\mrm{ha}(\vec v,w) = \begin{pmatrix}
        \nablah \Phi_h(v) \\ 0
    \end{pmatrix} = \begin{pmatrix}
        \nablah \Delta_h^{-1} \dvh \overline {\vec v} \\0
    \end{pmatrix}, \end{aligned} \\
    & \label{wdpt-003} 
    \begin{aligned}
    \mathcal U_\mrm{va} & := \mathcal P^\varepsilon_\mrm{va}(q,\vec v,w) = \begin{pmatrix}
        \widetilde q \\ U_\mrm{va}
    \end{pmatrix} \\
    & \quad \text{with} \ U_\mrm{va} = \begin{pmatrix}
        \vec v_\mrm{va} \\ w_\mrm{va}
    \end{pmatrix} := P^\varepsilon_\mrm{va}(\vec v,w) = \begin{pmatrix}
        \varepsilon^2 \nablah \Phi_3(\vec v,w) \\ \dz \Phi_3(\vec v,w)
    \end{pmatrix} = \begin{pmatrix}
        \varepsilon^2 \nablah (\varepsilon^2 \Delta_h + \partial_{zz})^{-1} (\dvh \widetilde {\vec v} + \dz w) \\ \dz (\varepsilon^2 \Delta_h + \partial_{zz})^{-1} (\dvh \widetilde {\vec v} + \dz w)
    \end{pmatrix}, 
    \end{aligned}\\
    & \label{wdpt-004}
    \begin{aligned}
        \mathcal U_\sigma & := \mathcal P^\varepsilon_\sigma (q,\vec v,w) = \begin{pmatrix}
            m_q \\ U_\sigma
        \end{pmatrix}\\
        & \quad \text{with} \ U_\sigma = \begin{pmatrix}
            \vec v_\sigma \\ w_\sigma
        \end{pmatrix}:= \begin{pmatrix}
            \vec v \\ w
        \end{pmatrix} - U_\mrm{ha} - U_\mrm{va}.
    \end{aligned}
\end{align}
Then, one can obtain the following lemma:
\begin{lemma}
    \label{lm:wave-vector-norms}
    Thanks to \eqref{uniform-est-total-01}--\eqref{wdpt-004}, one has that uniformly-in-$\varepsilon $,
    \begin{equation}
        \label{wdpt-005}
        \begin{gathered}
        \overline q - m_q, \widetilde q, \vec v_\mrm{ha}, \vec v_\mrm{va}, \vec v_\sigma, \varepsilon w_\mrm{va} \in L^\infty(0,T;H^3)\cap L^2(0,T;H^4), \qquad m_q \in L^\infty(0,T)\cap L^2(0,T), \\
        \frac{\widetilde q}{\varepsilon}, \frac{\vec v_\mrm{va}}{\varepsilon} ,w_\mrm{va}, w_\sigma \in L^\infty(0,T;H^2) \cap L^2(0,T;H^3).
        \end{gathered}
    \end{equation}
In particular,
\begin{equation}\label{reg:U}
         \mathcal U_\mrm{ha} \in L^\infty(0,T;H^3)\cap L^2(0,T;H^4),\quad \mathcal U_\mrm{va}, \ \mathcal U_\sigma \in L^\infty(0,T;H^2)\cap L^2(0,T;H^3).
\end{equation}
\end{lemma}

\begin{proof}
We start with the $q$ terms. From \eqref{uniform-est-total-01}, we have $q\in L^\infty(0,T;H^3)\cap L^2(0,T;H^4)$ and $\frac{\dz q}{\varepsilon}\in L^\infty(0,T;H^2)\cap L^2(0,T;H^3)$. Therefore $\overline q - m_q, \widetilde q\in L^\infty(0,T;H^3)\cap L^2(0,T;H^4)$. As $m_q$ is a function of time only, we have $m_q\in L^\infty(0,T)\cap L^2(0,T)$. 
In addition, as $\int_{-1}^1 \widetilde qdz = 0$, one can easily verify that 
\[\frac{\widetilde q}{\varepsilon}(x,z,t) = \int_{-1}^z \frac{\dz \widetilde q}{\varepsilon}(x,\tilde z,t) d\tilde z - \frac{1}{2} \int_{-1}^{1} \int_{-1}^s \frac{\dz \widetilde q}{\varepsilon}(x,\tilde z,t) d\tilde z \,ds .
\]
This leads to $\frac{\widetilde q}{\varepsilon} \in L^\infty(0,T;H^2)\cap L^2(0,T;H^3)$.

Next for the $\vec v$ terms, recall that $\vec v_\mrm{ha} = \nablah\Phi_h(\vec v)$ and $\vec v_\mrm{va}=\varepsilon^2 \nablah \Phi_3(\vec v,w)$. In terms of Fourier modes, we have
\[
\hat{\vec v}_{ha,\vec k} = \frac{\vec k_h}{|\vec k_h|^2}(\vec k_h\cdot \hat {\vec v}_{\vec k}), \quad \hat{\vec v}_{va,\vec k}= \frac{\varepsilon^2 \vec k_h}{\varepsilon^2 |\vec k_h|^2 + |k_3|^2}(\vec k_h\cdot \hat{\vec v}_{\vec k} + k_3 \hat{w}_{\vec k})=: \hat{\vec v}_{va,\vec k,1}+ \hat{\vec v}_{va,\vec k,2}.
\]
By direct estimate, we have
\begin{equation*}
    \begin{gathered}
|\hat{\vec v}_{ha,\vec k}| \leq |\hat{\vec v}_{\vec k}|, \quad |\hat{\vec v}_{va,\vec k,1}| \leq |\hat{\vec v}_{\vec k}|, \quad |\vec k_h \cdot \hat{\vec v}_{va,\vec k,2}| \leq |k_3 \hat w_{\vec k}|,
\\
|k_3  \hat{\vec v}_{va,\vec k,2}|\leq |k_3 \hat w_{\vec k}| \frac{\varepsilon^2 |\vec k_h|}{\varepsilon^2 \frac{|\vec k_h|^2}{|k_3|} + |k_3|} \leq |k_3 \hat w_{\vec k}| \frac{\varepsilon^2 |\vec k_h|}{2\varepsilon |\vec k_h|}\leq \frac{\varepsilon}{2} |k_3 \hat w_{\vec k}|,
\end{gathered}
\end{equation*}
where in the last step we used the fact that the minimum of $f(x) = x + \frac{a^2}{x}$ is $2|a|$ for $x>0$. From \eqref{uniform-est-total-01}, we know that $\vec v \in L^\infty(0,T;H^3)\cap L^2(0,T;H^4)$ and $\partial_z w \in L^\infty(0,T;H^2)\cap L^2(0,T;H^3)$. Consequently, we have $\vec v_\mrm{ha}, \vec v_\mrm{va} \in L^\infty(0,T;H^3)\cap L^2(0,T;H^4)$. $\vec v_{\sigma}=\vec v- \vec v_\mrm{ha} - \vec v_\mrm{va}$ has the same bound. In addition, as 
\begin{align}\label{eps-bound-alg}
    \frac{\varepsilon |\vec k_h|}{\varepsilon^2 |\vec k_h|^2 + |k_3|^2} = \frac{1}{\varepsilon|\vec k_h| + \frac{k_3^2}{\varepsilon|\vec k_h|}} \leq \frac{1}{2|k_3|} \leq 1,
\end{align}
one has that  $|\frac{1}{\varepsilon}\hat{\vec v}_{va,\vec k}| \leq |\hat{\vec v}_{\vec k}\cdot \vec k_h| + |k_3 \hat w_{\vec k}|$. Therefore, $\frac{\vec v_\mrm{va}}{\varepsilon} \in L^\infty(0,T;H^2)\cap L^2(0,T;H^3)$.

Finally, for the $w$ terms, recall that $w_\mrm{va} = \dz \Phi_3$ and $w_\sigma = w- w_\mrm{va}$. In terms of Fourier modes, 
\[
|\hat w_{va,\vec k}|= \left|\frac{k_3}{\varepsilon^2 |\vec k_h|^2 + |k_3|^2}(\hat{\vec v}_{\vec k}\cdot \vec k_h + k_3 \hat w_{\vec k})\right| \leq |\hat{\vec v}_{\vec k}\cdot \vec k_h| + |k_3 \hat w_{\vec k}|
\]
since $|k_3|\geq 1$, and
\[
|\varepsilon\hat w_{va,\vec k}|= \left|\frac{\varepsilon k_3}{\varepsilon^2 |\vec k_h|^2 + |k_3|^2}(\hat{\vec v}_{\vec k}\cdot \vec k_h + k_3 \hat w_{\vec k})\right| \leq \frac12|\hat{\vec v}_{\vec k}| + |\varepsilon \hat w_{\vec k}|
\]
thanks to \eqref{eps-bound-alg}.
Therefore, $w_\mrm{va}, w_\sigma \in L^\infty(0,T;H^2)\cap L^2(0,T;H^3)$ and $\varepsilon w_\mrm{va} \in L^\infty(0,T;H^3)\cap L^2(0,T;H^4)$.

\end{proof}

\begin{lemma}\label{lemma:Lva-decomposition}
    Any vector field $\varphi=(\varphi_q, \varphi_v, \varphi_w)^\top: \mathbb T^2\times 2\mathbb T \mapsto \mathbb R \times \mathbb R^2 \times \mathbb R $ can be decomposed as 
    \begin{equation}
    \label{def:va-decomposition}
      \varphi = {\varphi_\mrm{va,ker}}, 
      +  {\varphi_\mrm{va,\perp}}, 
    \end{equation}
    where $ \varphi_\mrm{va,ker} := \mathcal Q_\mrm{va}^\mrm{ker} \varphi \in \mathrm{ker} \mathcal L_\mrm{va}^\varepsilon$ and $ \varphi_\mrm{va,\perp} := \mathcal Q_\mrm{va}^{\perp} \varphi \in (\mathrm{ker} \mathcal L_\mrm{va}^\varepsilon)^\perp $ with respect to the $ L^2 $-inner product. Here $ \mathcal Q_\mrm{va}^\mrm{ker} $ and $ \mathcal Q_\mrm{va}^\perp $ are the corresponding $L^2$-orthogonal projectors. 
    Moreover, it holds that
    \begin{equation}\label{bdd-Lvaepsinverse}
        \norm{(\mathcal L_\mrm{va}^\varepsilon)^{-1} \mathcal Q_\mrm{va}^\perp \widetilde \varphi}{H^s} \lesssim \norm{\widetilde\varphi}{H^s}, \qquad \norm{(\mathcal L_\mrm{va}^\varepsilon)^{-1} \mathcal Q_\mrm{va}^\perp \overline \varphi}{H^s} \lesssim \frac{1}{\varepsilon} \norm{\overline\varphi}{H^s}.
    \end{equation}
    
\end{lemma}

\begin{proof}
Recall that $ \mathcal L_\mrm{va}^\varepsilon $ in \eqref{def:vertical-aw-operator} is given by
\begin{equation}
    \mathcal L_\mrm{va}^\varepsilon \varphi = \begin{pmatrix}
        \varepsilon \dvh\varphi_v + \dz \varphi_w \\
        \varepsilon \nablah \varphi_q \\ \dz \varphi_q
    \end{pmatrix}.
\end{equation}
    Define
    \begin{align}
    \label{def:projection-Q}
    \varphi_\mrm{va,\perp} = \mathcal Q_\mrm{va}^\perp \varphi & := \begin{pmatrix}
        \varphi_q - {m_{\varphi_q}} \\ \varepsilon \nablah \Psi \\
        \dz \Psi
    \end{pmatrix},\qquad 
    \varphi_\mrm{va,ker} = \mathcal Q_\mrm{va}^\mrm{ker} \varphi = \varphi - \varphi_\mrm{va,\perp} 
     :=
    \begin{pmatrix}
        m_{\varphi_q}
        \\
        \frac{1}{\varepsilon}\psi_{v} 
        \\
        \psi_w
    \end{pmatrix}
    \end{align}
    where $\Psi=(\varepsilon^2 \Delta_h + \partial_{zz})^{-1} (\varepsilon \dvh \varphi_v + \dz \varphi_w )$, $\psi_v = \varepsilon(\varphi_v - \varepsilon\nabla_h \Psi)$, and $\psi_w = \varphi_w - \dz \Psi$. One can verify that $\divh \psi_v + \dz \psi_w=0$, which leads to $\varphi_\mrm{va,ker} \in \mathrm{ker} \mathcal L_\mrm{va}^\varepsilon$, and thus $\varphi_\mrm{va,\perp} = \varphi-\varphi_\mrm{va,ker} \in (\mathrm{ker} \mathcal L_\mrm{va}^\varepsilon)^{\perp}$. By direct calculation it follows that $\mathcal Q_\mrm{va}^\mrm{ker} $ and $ \mathcal Q_\mrm{va}^\perp $ are $L^2$-orthogonal projections.

    Now we compute 
    \begin{align*}
        (\mathcal L_\mrm{va}^\varepsilon)^{-1} \mathcal Q_\mrm{va}^\perp \varphi = (\mathcal L_\mrm{va}^\varepsilon)^{-1} \varphi_\mrm{va,\perp} =
        \begin{pmatrix}
            (\varepsilon^2 \Delta_h + \partial_{zz})^{-1} (\varepsilon \dvh \varphi_v + \dz \varphi_w) 
            \\
            \varepsilon\nablah (\varepsilon^2 \Delta_h + \partial_{zz})^{-1}(\varphi_q - m_{\varphi_q})
            \\
            \dz (\varepsilon^2 \Delta_h + \partial_{zz})^{-1}(\varphi_q - m_{\varphi_q})
        \end{pmatrix}.
    \end{align*}
    Similar to the proof in Lemma~\ref{lm:wave-vector-norms}, one obtains that \eqref{bdd-Lvaepsinverse} holds.
    
\end{proof}

Now, we rewrite system \eqref{sys:hst-bsnq-sym} as the following:
\begin{equation}
    \label{sys:wdpt-form}
    \partial_t \begin{pmatrix}
        q \\ \vec v \\ w
    \end{pmatrix} + \begin{pmatrix}
        \frac{1}{\varepsilon} (\dvh \vec v + \dz w) \\ \frac{1}{\varepsilon} \nablah q \\ \frac{1}{\varepsilon^3} \dz q
    \end{pmatrix} = S =: \begin{pmatrix}
        S_1 \\ S_2 \\ S_3
    \end{pmatrix}
\end{equation}
where
\begin{equation}
    \label{wdpt-101}
    \begin{aligned}
        S_1 & := w \dz g - \vec v \cdot \nablah q - w \dz q && \in L^\infty(0,T;H^2) \cap L^2(0,T;H^3), \\
        S_2 & := \frac{1}{\rho} \Delta \vec v - \vec v \cdot \nablah \vec v - w \dz \vec v && \in L^\infty(0,T;H^1) \cap L^2(0,T;H^2), \\
        S_3 & := \frac{1}{\rho} \Delta w - \vec v \cdot \nablah w - w \dz w, \quad \dz S_3 && \in L^\infty(0,T;L^2) \cap L^2(0,T;H^1),
    \end{aligned}
\end{equation}
uniformly-in-$\varepsilon$ 
thanks to \eqref{uniform-est-total-01}.

Applying $ \mathcal P_\mrm{ha}, \ \mathcal P^\varepsilon_\sigma $ to both sides of system \eqref{sys:wdpt-form} leads to
\begin{align}
    \label{wdpt-102}
    (\dt + \frac{1}{\varepsilon} \mathcal L_\mrm{ha}) \mathcal U_\mrm{ha} & = \mathcal P_\mrm{ha}(S)  && \in  L^\infty(0,T;H^1) \cap L^2(0,T;H^2), \\
    \label{wdpt-103}
    \dt \mathcal U_\sigma & = \mathcal P_\sigma^\varepsilon (S) &&  \in  L^\infty(0,T;L^2) \cap L^2(0,T;H^1), 
\end{align}
uniformly-in-$\varepsilon $, where $ \mathcal L_\mrm{ha} $ is the horizontal acoustic wave operator defined in \eqref{def:horizontal-aw-operator}. Here the regularity follows with similar arguments as in the proof of Lemma \ref{lm:wave-vector-norms}. Consequently, one has that, uniformly-in-$\varepsilon $, 
\begin{equation}
    \label{wdpt-1031}
    \dt ( e^{\frac{t}{\varepsilon} \mathcal L_\mrm{ha}} \mathcal U_\mrm{ha}) = e^{\frac{t}{\varepsilon} \mathcal L_\mrm{ha}} \mathcal P_\mrm{ha}(S) \in L^\infty(0,T;H^1) \cap L^2(0,T;H^2).
\end{equation}

Due to the anisotropic nature of the problem, the vertical acoustic wave is a bit involved. Write, thanks to Lemma \ref{lm:wave-vector-norms},
\begin{equation}
    \label{wdpt-104}
    \begin{gathered}
    A_\varepsilon:= \begin{pmatrix}
        \frac{1}{\varepsilon} & & \\
        & \frac{1}{\varepsilon} \mathbb I_2 & \\
        & &   1
    \end{pmatrix} \qquad \text{and} \qquad \mathcal U_{\mrm{va},\varepsilon} := A_\varepsilon \mathcal U_\mrm{va} = \begin{pmatrix}
        \frac{\widetilde q}{\varepsilon} \\ \frac{\vec v_\mrm{va}}{\varepsilon} \\ w_\mrm{va}
    \end{pmatrix} \in L^\infty(0,T;H^2) \cap L^2(0,T;H^3),
    \\
    \varepsilon \mathcal U_{\mrm{va},\varepsilon} \in L^\infty(0,T;H^3) \cap L^2(0,T;H^4),
    \end{gathered}
\end{equation}
uniformly-in-$\varepsilon $. 
Then one has, after applying $ A_\varepsilon \mathcal P^\varepsilon_\mrm{va} $ to system \eqref{sys:wdpt-form}, that
\begin{equation}
    \label{wdpt-105}
    (\dt + \frac{1}{\varepsilon^2} \mathcal L_\mrm{va}^\varepsilon) \mathcal U_{\mrm{va},\varepsilon} = A_\varepsilon \mathcal P_\mrm{va}^\varepsilon (S) \in \frac{1}{\varepsilon} \Big[L^\infty(0,T;L^2) \cap L^2(0,T;H^1) \Big],
\end{equation}
where the vertical acoustic wave operator $ \mathcal L^\varepsilon_\mrm{va} $ is defined in \eqref{def:vertical-aw-operator}. Therefore, one has that
\begin{equation}
    \label{wdpt-106}
    \dt(\varepsilon e^{\frac{t}{\varepsilon^2}\mathcal L_\mrm{va}^\varepsilon} \mathcal U_{\mrm{va},\varepsilon} ) = \varepsilon e^{\frac{t}{\varepsilon^2}\mathcal L_\mrm{va}^\varepsilon} A_\varepsilon \mathcal P_\mrm{va}^\varepsilon (S) \in L^\infty(0,T;L^2) \cap L^2(0,T;H^1),
\end{equation}
uniformly-in-$\varepsilon $.

Therefore, we have the following compactness:
\begin{proposition}[Compactness]
    \label{prop:compactness}
    There exist
    \begin{equation}
        \label{def:limit-variable}
        \begin{gathered}
            \mathsf q_\mrm{ha}, \mathsf v_\mrm{ha}, \mathsf v_\sigma 
            \in L^\infty(0,T;H^3) \cap L^2(0,T;H^4),  \\
             \mathsf m_q \in L^\infty(0,T)\cap L^2(0,T), \qquad \mathsf w_\sigma \in L^\infty(0,T;H^2) \cap L^2(0,T;H^3), 
        \end{gathered}
    \end{equation}
    with 
    \begin{equation}
        \label{time-regularity}
        \begin{gathered}
        \dt \mathsf U_\sigma \in L^\infty (0,T;L^2) \cap L^2(0,T;H^1), \quad \dt \mathsf U_\mrm{ha} \in L^\infty(0,T;H^1) \cap L^2(0,T;H^2),
        \end{gathered}
    \end{equation}
    where 
    \begin{equation}
        \label{limit-vectors}
        \mathsf U_\sigma:= \begin{pmatrix}
            \mathsf m_q \\ \mathsf v_\sigma \\ \mathsf w_\sigma
        \end{pmatrix}, \quad \mathsf U_\mrm{ha} := \begin{pmatrix}
            \mathsf q_\mrm{ha} \\ \mathsf v_\mrm{ha} \\ \mathsf 0
         \end{pmatrix},
    \end{equation}
    such that 
    \begin{align}
        \label{wdpt-201} 
        \mathcal U_\sigma & \rightarrow \mathsf U_\sigma  & \text{in} & \ C([0,T];H^1)\cap L^2(0,T;H^2), \\
        \label{wdpt-202}
        e^{\frac{t}{\varepsilon}\mathcal L_\mrm{ha}}\mathcal U_\mrm{ha} & \rightarrow \mathsf U_\mrm{ha} & \text{in} & \ C([0,T];H^2)\cap L^2(0,T;H^3), \\
        \label{wdpt-203}
        \varepsilon e^{\frac{t}{\varepsilon^2} \mathcal L_\mrm{va}^\varepsilon} \mathcal U_\mrm{va,\varepsilon} & \rightarrow 0  & \text{in} & \ C([0,T];H^2) \cap L^2(0,T;H^3).
    \end{align}
    Moreover, 
    \begin{align}
    \label{wdpt-203-01}
    \dt \mathcal U_\sigma & \rightharpoonup \dt \mathsf U_\sigma & \text{weakly in} & \ L^2(0,T;H^1), \\
    \label{wdpt-204}
        \mathcal U_\mrm{ha} & \rightharpoonup 0 & \text{weakly in} & \ L^2(0,T;H^4),
        \\
        \label{wdpt-205}
        \mathcal U_\mrm{ha} & \overset{*}{\rightharpoonup} 0 & \text{weak$*$-ly in} & \ L^\infty(0,T;H^3),\\
    \label{wdpt-206}
    \mathcal U_\mrm{va,\varepsilon}, \ \mathcal U_\sigma & \rightharpoonup 0, \ \mathsf U_\sigma & \text{weakly in} & \ L^2(0,T;H^3),\\
    \label{wdpt-207}
    \mathcal U_\mrm{va,\varepsilon},\ \mathcal U_\sigma & \overset{*}{\rightharpoonup} 0,\ \mathsf U_\sigma & \text{weak$*$-ly in} & \ L^\infty(0,T;H^2).
    \end{align}
\end{proposition}

\begin{proof}
    The strong convergences \eqref{wdpt-201}--\eqref{wdpt-203} follow from the Aubin-Lions compactness theorem. The weak and weak star convergences follow from the Banach-Alaoglu theorem. We will only prove the weak convergence to zero of $\mathcal U_\mrm{ha}$ and $\mathcal U_{\mrm{va},\varepsilon}$. We explain in details for $\mathcal U_{\mrm{va},\varepsilon}$, and the convergence of $\mathcal U_\mrm{ha}$ follows similarly. 

    Consider test functions $\psi_1 = \psi_1(t) \in C_c^\infty(0,T)$ and $\psi_2 = \psi_2(\vec x) \in C^\infty(\mathbb T^2\times 2\mathbb T)$. It is easy to verify that $\mathcal U_{\mrm{va},\varepsilon}\in (\mathrm{ker} \mathcal L_\mrm{va}^\varepsilon)^\perp$ and $ \overline{\mathcal U}_{\mrm{va},\varepsilon} = 0 $,
    thanks to \eqref{wdpt-003} and \eqref{wdpt-104}. 
    As the operator $e^{\frac{t}{\varepsilon^2}\mathcal L_\mrm{va}^\varepsilon}$ is $ H^s $-norm preserving as mentioned in \eqref{norm-presering-va}, one has the identity 
    \begin{equation}\label{id:july06-2026}
    \begin{aligned}
        &\int_0^T \int_{\mathbb T^2\times 2\mathbb T} \mathcal U_\mrm{va,\varepsilon} \psi _1 \psi_2 d\vec x dt = \int_0^T \int_{\mathbb T^2\times 2\mathbb T} \mathcal U_\mrm{va,\varepsilon} \psi _1 \widetilde{\psi}_2 d\vec x dt  = \int_0^T \int_{\mathbb T^2\times 2\mathbb T} \mathcal U_\mrm{va,\varepsilon} \psi _1 \mathcal Q_\mrm{va}^\perp \widetilde{\psi}_2 d\vec x dt
        \\
        = &\int_0^T \int_{\mathbb T^2\times 2\mathbb T}  \left(\varepsilon e^{\frac{t}{\varepsilon^2}\mathcal L_\mrm{va}^\varepsilon}\mathcal U_\mrm{va,\varepsilon}\right) \psi _1 \left( \frac{1}{\varepsilon} e^{\frac{t}{\varepsilon^2}\mathcal L_\mrm{va}^\varepsilon} \mathcal L_\mrm{va}^\varepsilon \underbrace{(\mathcal L_\mrm{va}^\varepsilon)^{-1}\mathcal Q_\mrm{va}^\perp\widetilde{\psi}_2}_{:=\psi_2^*}\right) d\vec x dt
        \\
        = &\int_0^T \int_{\mathbb T^2\times 2\mathbb T} \left(\varepsilon e^{\frac{t}{\varepsilon^2}\mathcal L_\mrm{va}^\varepsilon}\mathcal U_\mrm{va,\varepsilon}\right) \psi_1 \partial_t \left(\varepsilon  e^{\frac{t}{\varepsilon^2}\mathcal L_\mrm{va}^\varepsilon}\psi_2^*\right) d\vec x dt
        \\
        = &-\int_0^T \int_{\mathbb T^2\times 2\mathbb T} \dt \left(\varepsilon e^{\frac{t}{\varepsilon^2}\mathcal L_\mrm{va}^\varepsilon}\mathcal U_\mrm{va,\varepsilon} \psi_1\right) \left(\varepsilon  e^{\frac{t}{\varepsilon^2}\mathcal L_\mrm{va}^\varepsilon}\psi_2^*\right) d\vec x dt,
    \end{aligned}
    \end{equation}
    where we have used the following identity:
    \begin{equation}
        e^{\frac{t}{\varepsilon^2}\mathcal L_\mrm{va}^\varepsilon} ( \psi_1(t) \psi_2(\vec x) )= 
        e^{\frac{s}{\varepsilon^2}\mathcal L_\mrm{va}^\varepsilon} ( \psi_1(t) \psi_2(\vec x) )\vert_{s=t} = 
        \psi_1(t) e^{\frac{s}{\varepsilon^2}\mathcal L_\mrm{va}^\varepsilon}  \psi_2(\vec x) \vert_{s=t} = \psi_1(t) e^{\frac{t}{\varepsilon^2}\mathcal L_\mrm{va}^\varepsilon}  \psi_2(\vec x).
    \end{equation}
    We emphasize that $ \psi_2^* $ is well defined since $ \mathcal Q_\mrm{va}^\perp \widetilde{\psi}_2 \in (\ker \mathcal L_\mrm{va}^\varepsilon)^\perp $.
Moreover, thanks to the fact that $\dt \left(\varepsilon e^{\frac{t}{\varepsilon^2}\mathcal L_\mrm{va}^\varepsilon}\mathcal U_\mrm{va,\varepsilon} \psi_1\right)$ is uniformly-in-$\varepsilon$ bounded in $L^\infty(0,T;L^2) \cap L^2(0,T;H^1)$ from \eqref{wdpt-106} and that $\psi_2^*$ is uniformly bounded with respect to $\varepsilon$ from Lemma~\ref{lemma:Lva-decomposition}, we have that the integral in \eqref{id:july06-2026} converges to $0$ as $\varepsilon\to 0$, which shows that the weak limit of $ \mathcal U_\mrm{va,\varepsilon} $ is zero. Thus this finishes the proof. 
\end{proof}

\subsection{Convergence of the system and proof of Theorem \ref{thm:limit}}
\label{subsec:limit-sys}
Now we are ready to pass the limit $ \varepsilon \rightarrow 0 $ in system \eqref{sys:hst-bsnq-sym}. In fact, \eqref{hbsym-continuity} can be written as 
\begin{equation}
\label{limit:continuity}
    \dvh \vec v + \dz w = - \varepsilon (\dt q + \vec v \cdot \nablah q + w \dz q - w \dz g ) \rightharpoonup \dvh \mathsf v_\sigma + \dz \mathsf w_\sigma = 0,
\end{equation}
weakly in $ L^2(\mathbb T^2 \times 2 \mathbb T \times (0,T)) $, thanks to \eqref{uniform-est-total-01}, \eqref{wdpt-201}, and \eqref{wdpt-204}--\eqref{wdpt-207} in Proposition \ref{prop:compactness}.

On the other hand, from \eqref{hbsym-h-momentum}, one can write 
\begin{equation}
    \label{limit:h-momentum-01}
    \begin{gathered}
    L:= \dt \vec v_\sigma + \vec v_\sigma \cdot \nablah \vec v_\sigma + w_\sigma \dz \vec v_\sigma - \frac{1}{\rho} \Delta \vec v_\sigma  \\
    = \underbrace{- \vec v_\sigma \cdot \nablah (\vec v_\mrm{ha} + \vec v_\mrm{va}) - (\vec v_\mrm{ha} + \vec v_\mrm{va}) \cdot \nablah \vec v_\sigma - w_\mrm{va} \dz \vec v_\sigma - \nablah (\frac{\widetilde q}{\varepsilon}) +\frac{1}{\rho}\Delta (\vec v_\mrm{ha} + \vec v_\mrm{va})}_{=:R_1}  \\
    \underbrace{- \dt \vec v_\mrm{va} - \vec v_\mrm{va} \cdot \nablah (\vec v_\mrm{va} + \vec v_\mrm{ha}) -  \vec v_\mrm{ha} \cdot \nablah \vec v_\mrm{va} - (w_\sigma + w_\mrm{va}) \dz \vec v_\mrm{va}}_{=:R_2}  \\
    \underbrace{- \dt \vec v_\mrm{ha} - \vec v_\mrm{ha} \cdot \nablah \vec v_\mrm{ha} - \frac{1}{\varepsilon} \nablah \overline q}_{=:R_3}.
    \end{gathered}
\end{equation}
Thanks to \eqref{uniform-est-total-01}, \eqref{wdpt-203-01}, \eqref{wdpt-206}, \eqref{wdpt-207}, and \eqref{wdpt-201} in Proposition \ref{prop:compactness}, one can verify that 
\begin{equation}
    \label{limit:h-momentum-02}
    L \rightharpoonup \dt \mathsf v_\sigma + \mathsf v_\sigma \cdot \nablah \mathsf v_\sigma + \mathsf w_\sigma \dz \mathsf v_\sigma - \Delta \mathsf v_\sigma , \qquad \text{weakly in} \ L^2( 0,T;H^1( \mathbb T^2 \times 2 \mathbb T )).
\end{equation}
Meanwhile, thanks to \eqref{wdpt-201} and \eqref{wdpt-204}--\eqref{wdpt-207} in Proposition \ref{prop:compactness}, one can verify that 
\begin{equation}
    \label{limit:h-momentum-03}
    R_1 \rightharpoonup 0 \qquad \text{weakly in} \ L^2(0,T;H^1(\mathbb T^2 \times 2 \mathbb T)).
\end{equation}
On the other hand, 
thanks to \eqref{uniform-est-total-01} and \eqref{wdpt-005}, one has that
\begin{equation}
    \label{limit:h-momentum-04}
    R_2 \rightharpoonup 0 \qquad \text{weakly in the sense of distribution.} 
\end{equation}
Finally, notice that, from \eqref{wdpt-002}, $ \vec v_\mrm{ha} = \nablah \Phi_h(\vec v) = \nablah \Delta_h^{-1} \dvh \overline{\vec v} $, one can write
\begin{equation}
    \label{limit:h-momentum-05}
    R_3 = \nablah (\underbrace{-\dt \Phi_h(\vec v) - \frac{1}{2} \vert \nablah \Phi_h(\vec v) \vert^2 - \frac{\overline q}{\varepsilon}}_{=:- \Pi_\varepsilon}) \qquad \text{with} \quad \dz \Pi_\varepsilon = 0.
\end{equation}
Therefore, sending $ \varepsilon \rightarrow 0$ in \eqref{limit:h-momentum-05} implies that, there exists $\Pi $ such that 
\begin{equation}
    \label{limit:h-momentum-06}
    R_3 \rightharpoonup - \nablah \Pi \qquad \text{with} \quad \dz \Pi = 0.
\end{equation}
Consequently, from \eqref{limit:continuity}--\eqref{limit:h-momentum-06}, one can conclude that $ (\mathsf v_\sigma, \mathsf w_\sigma ) $ is a solution to the limit system \eqref{sys:hst-bsnq-intro}. This finishes the proof of Theorem \ref{thm:limit}.

\section{Some remarks on non-convergence}
\label{sec:non-cnvg}

\subsection{Inviscid case}
\label{subsec:ivc-non-cnvg}

In this section, we examine the inviscid version of system \eqref{sys:hst-bsnq-sym} to illustrate a possible obstruction to extending the uniform convergence result to the inviscid setting.
In particular, we consider the following inviscid system with $ g\equiv 0 $, 
\begin{subequations}
    \label{sys:hst-bsnq-sym-ivc}
    \begin{gather}
\label{hbsym-continuity-ivc}
\dt q + \vec v \cdot \nablah q +  w \dz q  + \frac{1}{\varepsilon} (\dvh \vec v + \dz w)  = 0,  \\
\label{hbsym-h-momentum-ivc}
\begin{gathered}
\dt \vec v + \vec v \cdot \nablah \vec v + w \dz \vec v + \frac{1}{\varepsilon} \nablah q 
= 0,
\end{gathered}\\
\label{hbsym-v-momentum-ivc}
\begin{gathered}
    \varepsilon^2  ( \dt w + \vec v \cdot \nablah w + w \dz w) + \frac{1}{\varepsilon} \dz q 
    = 0.
\end{gathered}
    \end{gather}
\end{subequations}
Moreover, notice that one can consider system \eqref{sys:hst-bsnq-sym-ivc} with $ (q,\vec v, w)^{\top} = (q, v + U, 0, w)^{\top}(x,z,t) $ where $ (q, v, w) = (q,v,w)(x,z,t), \ U = U(z) \in \mathbb R $. That is, we consider the special solution of system \eqref{sys:hst-bsnq-sym-ivc} with only one horizontal spatial dependence and one horizontal velocity component, near a background steady flow $ \vec v = (U(z), 0)^\top $. Then the linearized system of $ (q, v, w) $ is given by
\begin{subequations}
    \label{sys:ivs-linear}
    \begin{align}
        \label{ivs-ln-cnt}
        \dt q + U \partial_x q + \frac{1}{\varepsilon} (\partial_x v + \dz w) & = 0, \\
        \label{ivs-ln-h-mmt}
        \dt v + U \partial_x v + w \dz U + \frac{1}{\varepsilon} \partial_x q & = 0, \\
        \label{ivs-ln-v-mmt}
        \dt w + U \partial_x w + \frac{1}{\varepsilon^3} \dz q & = 0. 
    \end{align}
\end{subequations}
We look for solutions to system \eqref{sys:ivs-linear} of the form 
\begin{equation} 
\label{ivs-ansatz}
(q,v,w)^\top = e^{ik_1x + \lambda t} (\hat q,\hat v, \hat w)^\top(z). \end{equation} 
Moreover, let \begin{equation}\label{def:V-lambda} V_\lambda := \lambda + ik_1 U. \end{equation} 
Then one has that 
\begin{equation}
    \label{ivs-ln-modes}
    \begin{aligned}
        V_\lambda \hat q + \frac{1}{\varepsilon}(ik_1 \hat v + \hat w') & = 0, \\
        V_\lambda \hat v + \hat w U' + \frac{1}{\varepsilon} ik_1 \hat q & = 0, \\
        V_\lambda \hat w + \frac{1}{\varepsilon^3} \hat q' & = 0.
    \end{aligned}
\end{equation}
After a complicated but straightforward computation, one can calculate from \eqref{ivs-ln-modes} that
\begin{gather*}
    \hat w = - \frac{\hat q'}{\varepsilon^3 V_\lambda}, \ \hat w' = - \frac{\hat q''}{\varepsilon^3 V_\lambda} + \frac{\hat q' V_\lambda'}{\varepsilon^3 V_\lambda^2}, \ \hat v = - \frac{ik_1\hat q}{\varepsilon V_\lambda} - \frac{\hat w U'}{V_\lambda} = - \frac{ik_1 \hat q}{\varepsilon V_\lambda} + \frac{\hat q ' U'}{\varepsilon^3 V_\lambda^2}, \\
    - \varepsilon V_\lambda \hat q = ik_1\hat v + \hat w' = \frac{k_1^2 \hat q}{\varepsilon V_\lambda} - \frac{\hat q''}{\varepsilon^3 V_\lambda} + \frac{2 \hat q' V_\lambda'}{\varepsilon^3 V_\lambda^2},
\end{gather*}
and, therefore, one has that
\begin{equation}
    \label{ivs-ln-001}
    \bigl(\frac{\hat q}{V_\lambda}\bigr)'' = L_{\varepsilon,k_1} \bigl( \frac{\hat q}{V_\lambda}\bigr),
\end{equation}
where
\begin{equation}
    \label{ivs-ln-002}
    L_{\varepsilon,k_1} := V_{\lambda} \bigl( \frac{1}{V_\lambda} \bigr)''+ \varepsilon^4 V_\lambda^2 + \varepsilon^2 k_1^2 .
\end{equation}
Next, we look for a special solution of \eqref{ivs-ln-001}--\eqref{ivs-ln-002} with the form
\begin{equation}
    \label{ivs-ln-003}
    L_{\varepsilon, k_1} = \varepsilon^4( a + i b) + \varepsilon^2 k_1^2 \qquad \text{with} \ a,b\in \mathbb R. 
\end{equation}
Then one has that, 
\begin{equation}
    \label{ivs-ln-004}
    \varepsilon^4 (a + ib) = \frac{1}{V_\lambda^2} \bigl( - V_\lambda V_\lambda '' + 2 (V_\lambda')^2 + \varepsilon^4 V_\lambda^4 \bigr).
\end{equation}
In particular, suppose that $ \dz U = 0 $ i.e., $ U $ is a constant, then \eqref{ivs-ln-004} can be reduced to, recalling \eqref{def:V-lambda},
\begin{equation}
    \label{ivs-ln-005}
    a+ i b = V_\lambda^2= \lambda^2 - k_1^2 U^2 + i(2\lambda k_1 U).
\end{equation}
Now we choose a special solution to \eqref{ivs-ln-005}:
\begin{equation}
\label{sol:mode}
    \lambda = \frac{1}{\varepsilon}, \ U = 1, \ a = \frac{1}{\varepsilon^2} - k_1^2, \ b = \frac{2 k_1}{\varepsilon}.
\end{equation}
Then from \eqref{ivs-ln-003}, one can verify that
\begin{equation}
    L_{\varepsilon,k_1}  = \mathcal O(\varepsilon^2),
\end{equation}
from which one can find the corresponding $ (\hat q, \hat v, \hat w) $ from  \eqref{ivs-ln-modes}--\eqref{ivs-ln-001}. From \eqref{ivs-ansatz} and \eqref{sol:mode}, one can conclude that this leads to the instability as $ \varepsilon \rightarrow 0 $ since $ \lambda \rightarrow \infty $.

\subsection{Gravity potential}
\label{subsec:non-prd-non-cnvg}

In this section, we demonstrate that with arbitrary gravity potential, the convergence result may not hold for system \eqref{sys:hst-bsnq-sym}. In particular, we consider the following system with $ \dz g = N \in \mathbb R $ and $ N \neq 0 $: 
\begin{subequations}
    \label{sys:gp-sys}
    \begin{gather}
    \label{hbsym-continuity-gp}
    \dt q + \vec v \cdot \nablah q +  w \dz q  + \frac{1}{\varepsilon} (\dvh \vec v + \dz w)  =   w N,  \\
    \label{hbsym-h-momentum-gp}
    \begin{gathered}
    \dt \vec v + \vec v \cdot \nablah \vec v + w \dz \vec v + \frac{1}{\varepsilon} \nablah q 
    = \frac{1}{\rho} \Delta \vec v,
    \end{gathered}\\
    \label{hbsym-v-momentum-gp}
    \begin{gathered}
        \varepsilon^2  ( \dt w + \vec v \cdot \nablah w + w \dz w) + \frac{1}{\varepsilon} \dz q 
        = \frac{\varepsilon^2}{\rho} \Delta w.
    \end{gathered}
    \end{gather}
\end{subequations}
Moreover, we consider the special case that $ \vec v \equiv  0 $ and $ (q,w) = (q,w)(z,t) $, i.e., only the vertical dynamics and vertical velocity remain. 
Then the leading order linearized system of \eqref{sys:gp-sys} yields 
\begin{subequations}
    \label{sys:gp-ln-sys}
    \begin{align}
        \label{eq:gp-ln-cnt}
    \dt q + \frac{1}{\varepsilon} \dz w & = w N, \\
    \label{eq:gp-ln-mmt}
    \dt w + \frac{1}{\varepsilon^3} \dz q & = \partial_{zz} w.
    \end{align}
\end{subequations}
We look for solutions to system \eqref{sys:gp-ln-sys} of the form $ (q, w) = e^{ik_3 z+\lambda t} (\hat q, \hat w) $ with $ \hat q, \hat w \in \mathbb C $. Then one has that
\begin{equation}
    \label{gp-ln-001}
    \lambda \hat q + \frac{i k_3}{\varepsilon} \hat w = N \hat w , \qquad \lambda \hat w + \frac{i k_3}{\varepsilon^3} \hat q + k_3^2 \hat w = 0.
\end{equation}
Eliminating $ \hat q $ and $ \hat w $ in \eqref{gp-ln-001} yields 
\begin{equation}
    \label{gp-ln-002}
    \varepsilon^4 \lambda^2 + \varepsilon^4  k_3^2 \lambda + k_3^2 + ik_3 \varepsilon N = 0.
\end{equation}
Thus,
\begin{equation}
    \begin{gathered}
    \lambda_\pm  = \frac{-\varepsilon^4 k_3^2 \pm \sqrt{\varepsilon^8 k_3^4 - 4 \varepsilon^4 k_3^2 - i 4\varepsilon^5 N k_3 }}{2\varepsilon^4}= -\frac{k_3^2}{2} \pm \sqrt{\frac{k_3^4}{4} - \frac{k_3^2}{\varepsilon^4} - i\frac{Nk_3}{\varepsilon^3}}. 
    \end{gathered}
\end{equation}
In particular, for $ k_3 = 1 $, one has that, 
\begin{equation}
    \begin{gathered}
        \lambda_\pm\vert_{k_3=1} = -\frac{1}{2} \pm \sqrt{\frac{1}{4} - \frac{1}{\varepsilon^4} - i\frac{N}{\varepsilon^3}} = - \frac{1}{2} \pm \left(-\frac{N}{2\varepsilon} + i(\frac{1}{\varepsilon^2} + \frac{N^2}{8}) \right)+ \mathcal O(\varepsilon). 
    \end{gathered}
\end{equation}
Therefore, for $ \varepsilon $ small enough, 
\begin{equation}
    \label{gp-lambda-positive}
    \max_{\pm}  \mrm{Re} (\lambda_\pm\vert_{k_3=1}) = -\frac{1}{2} + \frac{\vert N \vert}{2\varepsilon}+ \mrm{l.o.t} > 0.
\end{equation}
This corresponds to the instability as $ \varepsilon \rightarrow 0 $ since $ \max_{\pm}  \mrm{Re} (\lambda_\pm\vert_{k_3=1}) \rightarrow \infty $. 

\section*{Acknowledgment}
Q.L. was partially supported by the Simons Foundation (SFI-MPS-TSM-00013384).

\bibliographystyle{plain}
\bibliography{references}

@book{feireislSingularLimitsThermodynamics2017,
    address = {Cham},
    title = {Singular {Limits} in {Thermodynamics} of {Viscous} {Fluids}},
    isbn = {978-3-319-63780-8},
    url = {http://link.springer.com/10.1007/978-3-319-63781-5},
    doi = {10.1007/978-3-319-63781-5},
    publisher = {Springer International Publishing},
    author = {Feireisl, Eduard and Novotn\'y, Anton\'in},
    year = {2017},
    note = {Series Title: Advances in Mathematical Fluid Mechanics
Publication Title: Springer},
}

@techreport{USATMO,
  author      = {{U.S. Committee on Extension to the Standard Atmosphere}},
  title       = {U.S. Standard Atmosphere, 1976},
  institution = {U.S. Government Printing Office},
  address     = {Washington, D.C.},
  year        = {1976}
}

@misc{OceanDepth,
  author      = {{National Ocean Service}},
  title       = {How deep is the ocean?},
  year        = {2026},
  month       = {April 17},
  url         = {https://oceanservice.noaa.gov/facts/oceandepth.html},
  note        = {Accessed: 2026-04-18}
}

@misc{GreatLakes,
  author      = {{Wikipedia contributors}},
  title       = {Great Lakes},
  year        = {2026},
  month       = {April 18},
  url         = {https://en.wikipedia.org/wiki/Great_Lakes},
  organization = {Wikipedia, The Free Encyclopedia},
  note        = {Accessed: 2026-04-18}
}

@article{Cao2007,
    title = {Global well-posedness of the three-dimensional viscous primitive equations of large scale ocean and atmosphere dynamics},
    volume = {166},
    issn = {0003-486X},
    url = {http://annals.math.princeton.edu/2007/166-1/p07},
    doi = {10.4007/annals.2007.166.245},
    number = {1},
    journal = {Annals of Mathematics},
    author = {Cao, Chongsheng and Titi, Edriss},
    month = jul,
    year = {2007},
    note = {arXiv: math/0503028},
    pages = {245--267},
}

@article{liuRigorousJustificationHydrostatic2024,
    title = {Rigorous justification of the hydrostatic approximation limit of viscous compressible flows},
    volume = {464},
    issn = {01672789},
    url = {https://linkinghub.elsevier.com/retrieve/pii/S0167278924001465},
    doi = {10.1016/j.physd.2024.134195},
    language = {en},
    urldate = {2024-06-19},
    journal = {Physica D: Nonlinear Phenomena},
    author = {Liu, Xin and Titi, Edriss S.},
    month = aug,
    year = {2024},
    pages = {134195},
}

@article{liuZeroMachNumber2023,
    title = {Zero {Mach} number limit of the compressible primitive equations: {Ill}-prepared initial data},
    volume = {356},
    issn = {00220396},
    shorttitle = {Zero {Mach} number limit of the compressible primitive equations},
    url = {https://linkinghub.elsevier.com/retrieve/pii/S002203962300044X},
    doi = {10.1016/j.jde.2023.01.031},
    language = {en},
    urldate = {2023-05-21},
    journal = {Journal of Differential Equations},
    author = {Liu, Xin and Titi, Edriss S.},
    month = may,
    year = {2023},
    pages = {1--58},
}

@book{lionsMathematicalTopicsFluid1996,
    title = {Mathematical topics in fluid mechanics},
    author = {Lions, Pierre-Louis},
    month = jan,
    publisher = {Oxford University Press},
    year = {1996},
    note = {MAG ID: 1573547614},
}

@book{Lions1998,
    title = {Mathematical topics in fluid mechanics. {Volume} 2. {Compressible} models},
    isbn = {978-0-19-851488-6 0-19-851488-3},
    publisher = {Oxford University Press},
    author = {Lions, Pierre-Louis},
    year = {1998},
    note = {Series Title: Oxford Lecture Series in Mathematics and Its Applications , Vol 2, No 10
Publication Title: Oxford University Press},
}

@article{azerad2001mathematical,
	author = {Az{\'e}rad, Pascal and Guill{\'e}n, Francisco},
	journal = {SIAM Journal on Mathematical Analysis},
	number = {4},
	pages = {847--859},
	publisher = {SIAM},
	title = {Mathematical justification of the hydrostatic approximation in the primitive equations of geophysical fluid dynamics},
	volume = {33},
	year = {2001}}

@article{furukawa2020rigorous,
	author = {Furukawa, Ken and Giga, Yoshikazu and Hieber, Matthias and Hussein, Amru and Kashiwabara, Takahito and Wrona, Marc},
	journal = {Nonlinearity},
	number = {12},
	pages = {6502},
	publisher = {IOP Publishing},
	title = {Rigorous justification of the hydrostatic approximation for the primitive equations by scaled {N}avier--{S}tokes equations},
	volume = {33},
	year = {2020}}

@article{li2019primitive,
	author = {Li, Jinkai and Titi, Edriss S},
	journal = {Journal de Math{\'e}matiques Pures et Appliqu{\'e}es},
	pages = {30--58},
	publisher = {Elsevier},
	title = {The primitive equations as the small aspect ratio limit of the {N}avier--{S}tokes equations: Rigorous justification of the hydrostatic approximation},
	volume = {124},
	year = {2019}}

@article{li2022primitive,
	author = {Li, Jinkai and Titi, Edriss S and Yuan, Guozhi},
	journal = {Journal of Differential Equations},
	pages = {492--524},
	publisher = {Elsevier},
	title = {The primitive equations approximation of the anisotropic horizontally viscous {3D} {N}avier--{S}tokes equations},
	volume = {306},
	year = {2022}}

@article{blumen1972geostrophic,
	author = {Blumen, William},
	journal = {Reviews of Geophysics},
	number = {2},
	pages = {485--528},
	publisher = {Wiley Online Library},
	title = {Geostrophic adjustment},
	volume = {10},
	year = {1972}}

@article{gill1976adjustment,
	author = {Gill, Adrian E},
	journal = {Journal of Fluid Mechanics},
	number = {3},
	pages = {603--621},
	publisher = {Cambridge University Press},
	title = {Adjustment under gravity in a rotating channel},
	volume = {77},
	year = {1976}}

@book{gill1982atmosphere,
	author = {Gill, Adrian E},
	publisher = {Elsevier},
	title = {Atmosphere-ocean dynamics},
	year = {2016}}

@article{hermann1993energetics,
	author = {Hermann, AJ and Owens, WB},
	journal = {Journal of Physical Oceanography},
	number = {2},
	pages = {346--371},
	title = {Energetics of gravitational adjustment for mesoscale chimneys},
	volume = {23},
	year = {1993}}

@article{holton1973introduction,
	author = {Holton, James R},
	journal = {American Journal of Physics},
	number = {5},
	pages = {752--754},
	publisher = {American Association of Physics Teachers},
	title = {An introduction to dynamic meteorology},
	volume = {41},
	year = {1973}}

@article{kuo1997time,
	author = {Kuo, Allen C and Polvani, Lorenzo M},
	journal = {Journal of Physical Oceanography},
	number = {8},
	pages = {1614--1634},
	title = {Time-dependent fully nonlinear geostrophic adjustment},
	volume = {27},
	year = {1997}}

@article{plougonven2005lagrangian,
	author = {Plougonven, R and Zeitlin, V},
	journal = {Geophysical \& Astrophysical Fluid Dynamics},
	number = {2},
	pages = {101--135},
	publisher = {Taylor \& Francis},
	title = {Lagrangian approach to geostrophic adjustment of frontal anomalies in a stratified fluid},
	volume = {99},
	year = {2005}}

@article{rossby1938mutual,
	author = {Rossby, Carl-Gustav},
	journal = {Journal of Marine Research},
	number = {3},
	pages = {239--263},
	title = {On the mutual adjustment of pressure and velocity distributions in certain simple current systems, {II}},
	volume = {1},
	year = {1938}}

@article{kobelkov2006existence,
	author = {Kobelkov, Georgij M},
	journal = {Comptes Rendus Mathematique},
	number = {4},
	pages = {283--286},
	publisher = {Elsevier},
	title = {Existence of a solution `in the large' for the {3D} large-scale ocean dynamics equations},
	volume = {343},
	year = {2006}}

@article{kukavica2007regularity,
	author = {Kukavica, Igor and Ziane, Mohammed},
	journal = {Nonlinearity},
	number = {12},
	pages = {2739},
	publisher = {IOP Publishing},
	title = {On the regularity of the primitive equations of the ocean},
	volume = {20},
	year = {2007}}

@article{hieber2016global,
	author = {Hieber, Matthias and Kashiwabara, Takahito},
	journal = {Archive for Rational Mechanics and Analysis},
	number = {3},
	pages = {1077--1115},
	publisher = {Springer},
	title = {Global Strong Well-Posedness of the Three Dimensional Primitive Equations in {$L^p$}-Spaces},
	volume = {221},
	year = {2016}}

@article{brenier1999homogeneous,
	author = {Brenier, Yann},
	journal = {Nonlinearity},
	number = {3},
	pages = {495},
	publisher = {IOP Publishing},
	title = {Homogeneous hydrostatic flows with convex velocity profiles},
	volume = {12},
	year = {1999}}

@article{brenier2003remarks,
	author = {Brenier, Yann},
	journal = {Bulletin des Sciences Mathematiques},
	number = {7},
	pages = {585--595},
	publisher = {Elsevier},
	title = {Remarks on the derivation of the hydrostatic {E}uler equations},
	volume = {127},
	year = {2003}}

@article{grenier1999derivation,
	author = {Grenier, Emmanuel},
	journal = {ESAIM: Mathematical Modelling and Numerical Analysis},
	number = {5},
	pages = {965--970},
	publisher = {EDP Sciences},
	title = {On the derivation of homogeneous hydrostatic equations},
	volume = {33},
	year = {1999}}

@article{masmoudi2012h,
	author = {Masmoudi, Nader and Wong, Tak Kwong},
	journal = {Archive for Rational Mechanics and Analysis},
	number = {1},
	pages = {231--271},
	publisher = {Springer},
	title = {On the ${H}^s$ theory of hydrostatic {E}uler equations},
	volume = {204},
	year = {2012}}

@article{renardy2009ill,
	author = {Renardy, Michael},
	journal = {Archive for Rational Mechanics and Analysis},
	number = {3},
	pages = {877--886},
	publisher = {Springer},
	title = {Ill-posedness of the hydrostatic {E}uler and {N}avier--{S}tokes equations},
	volume = {194},
	year = {2009}}

@article{han2016ill,
	author = {Han-Kwan, Daniel and Nguyen, Toan T},
	journal = {Archive for Rational Mechanics and Analysis},
	number = {3},
	pages = {1317--1344},
	publisher = {Springer},
	title = {Ill-posedness of the hydrostatic {E}uler and singular {V}lasov equations},
	volume = {221},
	year = {2016}}

@article{collot2023stable,
	author = {Collot, Charles and Ibrahim, Slim and Lin, Quyuan},
	journal = {Annales de l'Institut Henri Poincar{\'e} C},
	number = {2},
	pages = {317--356},
	title = {Stable singularity formation for the inviscid primitive equations},
	volume = {41},
	year = {2023}}

@article{cao2015finite,
	author = {Cao, Chongsheng and Ibrahim, Slim and Nakanishi, Kenji and Titi, Edriss S},
	journal = {Communications in Mathematical Physics},
	number = {2},
	pages = {473--482},
	publisher = {Springer},
	title = {Finite-time blowup for the inviscid primitive equations of oceanic and atmospheric dynamics},
	volume = {337},
	year = {2015}}

@article{wong2015blowup,
	author = {Wong, Tak Kwong},
	journal = {Proceedings of the American Mathematical Society},
	number = {3},
	pages = {1119--1125},
	title = {Blowup of solutions of the hydrostatic {E}uler equations},
	volume = {143},
	year = {2015}}

@article{ibrahim2026profile,
  title={On the profile of singularity formation for the incompressible hydrostatic Boussinesq system},
  author={Ibrahim, Slim and Lin, Quyuan and Qian, Lingjun and Titi, Edriss S},
  journal={Nonlinearity},
  volume={39},
  number={4},
  pages={045015},
  year={2026},
  publisher={IOP Publishing}
}

@article{ibrahim2021finite,
	author = {Ibrahim, Slim and Lin, Quyuan and Titi, Edriss S},
	journal = {Journal of Differential Equations},
	pages = {557--577},
	publisher = {Elsevier},
	title = {Finite-time blowup and ill-posedness in {S}obolev spaces of the inviscid primitive equations with rotation},
	volume = {286},
	year = {2021}}

@article{hoffGlobalSolutionsNavierStokes1995a,
    title = {Global {Solutions} of the {Navier}-{Stokes} {Equations} for {Multidimensional} {Compressible} {Flow} with {Discontinuous} {Initial} {Data}},
    volume = {120},
    issn = {0022-0396},
    url = {https://www.sciencedirect.com/science/article/pii/S0022039685711114},
    doi = {10.1006/jdeq.1995.1111},
    number = {1},
    urldate = {2026-07-02},
    journal = {Journal of Differential Equations},
    author = {Hoff, D.},
    month = jul,
    year = {1995},
    pages = {215--254},
}

@article{Feireisl2001,
    title = {On the {Existence} of {Globally} {Defined} {Weak} {Solutions} to the {Navier}-{Stokes} {Equations}},
    volume = {3},
    issn = {1422-6928},
    doi = {10.1007/PL00000976},
    number = {4},
    journal = {Journal of Mathematical Fluid Mechanics},
    author = {Feireisl, E. and Novotn\'y, Anton\'in and Petzeltov\'a, H.},
    year = {2001},
    pages = {358--392},
}

@article{Hoff2012,
    title = {Local {Solutions} of a {Compressible} {Flow} {Problem} with {Navier} {Boundary} {Conditions} in {General} {Three}-{Dimensional} {Domains}},
    volume = {44},
    issn = {0036-1410},
    url = {http://epubs.siam.org/doi/abs/10.1137/110827065},
    doi = {10.1137/110827065},
    number = {2},
    journal = {SIAM Journal on Mathematical Analysis},
    author = {Hoff, David},
    month = jan,
    year = {2012},
    pages = {633--650},
}

@article{Itaya1971,
    title = {On the {Cauchy} {Problems} for the {System} of {Fundamental} {Equations} {Describing} the {Movement} of {Compressible} {Viscous} {Fluid}},
    volume = {23},
    journal = {Kodai Math. Sem. Rep.},
    author = {Itaya, Nobutoshi},
    year = {1971},
    pages = {60--120},
}

@article{Klainerman1981,
    title = {Singular limits of quasilinear hyperbolic systems with large parameters and the incompressible limit of compressible fluids},
    volume = {34},
    issn = {00103640},
    url = {http://doi.wiley.com/10.1002/cpa.3160340405},
    doi = {10.1002/cpa.3160340405},
    number = {4},
    journal = {Communications on Pure and Applied Mathematics},
    author = {Klainerman, Sergiu and Majda, Andrew},
    month = jul,
    year = {1981},
    pages = {481--524},
}

@article{Klainerman1982,
    title = {Compressible and incompressible fluids},
    volume = {35},
    issn = {00103640},
    url = {http://doi.wiley.com/10.1002/cpa.3160350503},
    doi = {10.1002/cpa.3160350503},
    number = {5},
    journal = {Communications on Pure and Applied Mathematics},
    author = {Klainerman, Sergiu and Majda, Andrew},
    month = sep,
    year = {1982},
    pages = {629--651},
}

@article{Ukai1986,
    title = {The incompressible limit and the initial layer of the compressible {Euler} equation},
    volume = {26},
    issn = {0023-608X},
    url = {http://projecteuclid.org/euclid.kjm/1250520925},
    doi = {10.1215/kjm/1250520925},
    number = {2},
    journal = {Journal of Mathematics of Kyoto University},
    author = {Ukai, Seiji},
    year = {1986},
    pages = {323--331},
}

@article{Lions1998a,
    title = {Incompressible limit for a viscous compressible fluid},
    volume = {77},
    issn = {00217824},
    doi = {10.1016/S0021-7824(98)80139-6},
    number = {6},
    journal = {Journal des Mathematiques Pures et Appliquees},
    author = {Lions, Pierre-Louis and Masmoudi, Nader},
    year = {1998},
    pages = {585--627},
}

@article{Masmoudi2001,
    title = {Incompressible, inviscid limit of the compressible {Navier}–{Stokes} system},
    volume = {18},
    issn = {02941449},
    url = {https://linkinghub.elsevier.com/retrieve/pii/S0294144900001232},
    doi = {10.1016/S0294-1449(00)00123-2},
    number = {2},
    journal = {Annales de l'Institut Henri Poincare (C) Non Linear Analysis},
    author = {Masmoudi, Nader},
    month = mar,
    year = {2001},
    pages = {199--224},
}

@article{rajagopalOberbeckboussinesqApproximation1996,
    title = {On the oberbeck-boussinesq approximation},
    volume = {06},
    issn = {0218-2025},
    url = {https://www.worldscientific.com/doi/abs/10.1142/s0218202596000481},
    doi = {10.1142/S0218202596000481},
    number = {08},
    urldate = {2026-07-02},
    journal = {Mathematical Models and Methods in Applied Sciences},
    publisher = {World Scientific Publishing Co.},
    author = {Rajagopal, K.r. and Ruzicka, M. and Srinivasa, A.r.},
    month = dec,
    year = {1996},
    pages = {1157--1167},
}

@article{liuLocalWellPosednessStrong2021,
    title = {Local {Well}-{Posedness} of {Strong} {Solutions} to the {Three}-{Dimensional} {Compressible} {Primitive} {Equations}},
    volume = {241},
    issn = {0003-9527, 1432-0673},
    url = {https://link.springer.com/10.1007/s00205-021-01662-3},
    doi = {10.1007/s00205-021-01662-3},
    language = {en},
    number = {2},
    urldate = {2022-04-28},
    journal = {Archive for Rational Mechanics and Analysis},
    author = {Liu, Xin and Titi, Edriss S.},
    month = aug,
    year = {2021},
    pages = {729--764},
}

@article{LT2018b,
    title = {Global {Existence} of {Weak} {Solutions} to the {Compressible} {Primitive} {Equations} of {Atmospheric} {Dynamics} with {Degenerate} {Viscosities}},
    volume = {51},
    issn = {0036-1410},
    url = {https://epubs.siam.org/doi/10.1137/18M1211994},
    doi = {10.1137/18M1211994},
    number = {3},
    journal = {SIAM Journal on Mathematical Analysis},
    author = {Liu, Xin and Titi, Edriss S.},
    month = jan,
    year = {2019},
    pages = {1913--1964},
}

@article{LT2018LowMach1,
    title = {Zero {Mach} {Number} {Limit} of the {Compressible} {Primitive} {Equations}: {Well}-{Prepared} {Initial} {Data}},
    volume = {238},
    issn = {0003-9527},
    url = {http://arxiv.org/abs/1905.09367},
    doi = {10.1007/s00205-020-01553-z},
    number = {2},
    journal = {Archive for Rational Mechanics and Analysis},
    author = {Liu, Xin and Titi, Edriss S.},
    month = nov,
    year = {2020},
    note = {arXiv: 1905.09367},
    pages = {705--747},
}

@misc{necasovaEnergyEqualityCompressible2024,
    title = {Energy equality for the compressible {Primitive} {Equations} with vacuum},
    url = {http://arxiv.org/abs/2303.11129},
    doi = {10.48550/arXiv.2303.11129},
    urldate = {2025-01-01},
    publisher = {arXiv},
    author = {Ne\v{c}asov\'a, \v{S}\'arka and Rodriguez-Bellido, Maria Angeles and Tang, Tong},
    month = dec,
    year = {2024},
    note = {arXiv:2303.11129 [math]},
}

@article{tangDerivationInviscidCompressible2023,
    title = {Derivation of the inviscid compressible {Primitive} {Equations}},
    volume = {139},
    issn = {0893-9659},
    url = {https://www.sciencedirect.com/science/article/pii/S0893965922003974},
    doi = {10.1016/j.aml.2022.108534},
    language = {en},
    urldate = {2023-03-02},
    journal = {Applied Mathematics Letters},
    author = {Tang, Tong and Nečasová, Šárka},
    month = may,
    year = {2023},
    pages = {108534},
}

@article{Ersoy2012,
    title = {Existence of a global weak solution to {Compressible} {Primitive} {Equations}},
    volume = {350},
    issn = {1631073X},
    url = {http://dx.doi.org/10.1016/j.crma.2012.04.013},
    doi = {10.1016/j.crma.2012.04.013},
    number = {7-8},
    journal = {Comptes Rendus Mathematique},
    publisher = {Elsevier Masson SAS},
    author = {Ersoy, Mehmet and Ngom, Timack},
    month = apr,
    year = {2012},
    pages = {379--382},
}

@article{Ersoy2011a,
    title = {Compressible primitive equations: {Formal} derivation and stability of weak solutions},
    volume = {24},
    issn = {09517715},
    doi = {10.1088/0951-7715/24/1/004},
    number = {1},
    journal = {Nonlinearity},
    author = {Ersoy, Mehmet and Ngom, Timack and Sy, Mamadou},
    year = {2011},
    pages = {79--96},
}

@misc{hieberLagrangianApproachCompressible2025,
    title = {The {Lagrangian} approach to the compressible primitive equations},
    url = {http://arxiv.org/abs/2502.03630},
    doi = {10.48550/arXiv.2502.03630},
    urldate = {2025-02-09},
    publisher = {arXiv},
    author = {Hieber, Matthias and Iida, Yoshiki and Roy, Arnab and Zöchling, Tarek},
    month = feb,
    year = {2025},
    note = {arXiv:2502.03630 [math]},
}

@article{gaoHydrostaticApproximationCompressible2022,
    title = {On the {Hydrostatic} {Approximation} of {Compressible} {Anisotropic} {Navier}–{Stokes} {Equations}–{Rigorous} {Justification}},
    volume = {24},
    issn = {1422-6952},
    url = {https://doi.org/10.1007/s00021-022-00717-z},
    doi = {10.1007/s00021-022-00717-z},
    language = {en},
    number = {3},
    urldate = {2023-03-02},
    journal = {Journal of Mathematical Fluid Mechanics},
    author = {Gao, Hongjun and Nečasová, Šárka and Tang, Tong},
    month = jul,
    year = {2022},
    pages = {86},
}

@article{schochetSingularLimitsBounded1987,
    title = {Singular limits in bounded domains for quasilinear symmetric hyperbolic systems having a vorticity equation},
    volume = {68},
    issn = {0022-0396},
    url = {https://www.sciencedirect.com/science/article/pii/0022039687901781},
    doi = {10.1016/0022-0396(87)90178-1},
    number = {3},
    urldate = {2024-04-21},
    journal = {Journal of Differential Equations},
    author = {Schochet, Steven},
    month = jul,
    year = {1987},
    pages = {400--428},
}

@article{schochetUniformExistenceConvergence2022,
    title = {Toward uniform existence and convergence theorems for three-scale systems of hyperbolic {PDEs} with general initial data},
    issn = {0360-5302, 1532-4133},
    url = {https://www.tandfonline.com/doi/full/10.1080/03605302.2022.2129383},
    doi = {10.1080/03605302.2022.2129383},
    language = {en},
    urldate = {2022-10-29},
    journal = {Communications in Partial Differential Equations},
    author = {Schochet, Steve and Xu, Xin},
    month = oct,
    year = {2022},
    pages = {1--43},
}

@article{SchochetCMP1986,
    title = {The compressible {Euler} equations in a bounded domain: {Existence} of solutions and the incompressible limit},
    volume = {104},
    issn = {0010-3616},
    url = {http://link.springer.com/10.1007/BF01210792},
    doi = {10.1007/BF01210792},
    number = {1},
    journal = {Communications in Mathematical Physics},
    author = {Schochet, Steve},
    month = mar,
    year = {1986},
    pages = {49--75},
}

@article{Gallagher1998,
    title = {Applications of {Schochet}'s methods to parabolic equations},
    volume = {77},
    url = {http://epubs.siam.org/doi/10.1137/15M1018253},
    doi = {10.1016/S0021-7824(99)80002-6},
    number = {10},
    journal = {Journal de Mathématiques Pures et Appliquées},
    author = {Gallagher, I.},
    month = dec,
    year = {1998},
    pages = {989--1054},
}

@article{Schochet1994,
    title = {Fast {Singular} {Limits} of {Hyperbolic} {PDEs}},
    volume = {114},
    issn = {00220396},
    url = {http://www.sciencedirect.com/science/article/pii/S0022039684711570},
    doi = {10.1006/jdeq.1994.1157},
    number = {2},
    journal = {Journal of Differential Equations},
    author = {Schochet, Steve},
    month = dec,
    year = {1994},
    pages = {476--512},
}

@article{embidAveragingFastGravity1996,
    title = {Averaging over fast gravity waves for geophysical flows with arbitary potential vorticity},
    volume = {21},
    issn = {0360-5302, 1532-4133},
    url = {https://www.tandfonline.com/doi/full/10.1080/03605309608821200},
    doi = {10.1080/03605309608821200},
    language = {en},
    number = {3-4},
    urldate = {2022-05-03},
    journal = {Communications in Partial Differential Equations},
    author = {Embid, Pedro F. and Majda, Andrew J.},
    month = jan,
    year = {1996},
    pages = {619--658},
}

@article{majdaAveragingFastGravity1998,
    title = {Averaging over {Fast} {Gravity} {Waves} for {Geophysical} {Flows} with {Unbalanced} {Initial} {Data}},
    volume = {11},
    issn = {0935-4964, 1432-2250},
    url = {http://link.springer.com/10.1007/s001620050086},
    doi = {10.1007/s001620050086},
    language = {en},
    number = {3-4},
    urldate = {2022-05-03},
    journal = {Theoretical and Computational Fluid Dynamics},
    author = {Majda, Andrew J. and Embid, Pedro},
    month = jun,
    year = {1998},
    pages = {155--169},
}

@article{embidLowFroudeNumber1998,
    title = {Low {Froude} number limiting dynamics for stably stratified flow with small or finite {Rossby} numbers},
    volume = {87},
    issn = {0309-1929, 1029-0419},
    url = {http://www.tandfonline.com/doi/abs/10.1080/03091929808208993},
    doi = {10.1080/03091929808208993},
    language = {en},
    number = {1-2},
    urldate = {2022-06-22},
    journal = {Geophysical and Astrophysical Fluid Dynamics},
    author = {Embid, Pedro F. and Majda, Andrew J.},
    month = mar,
    year = {1998},
    pages = {1--50},
}

@article{Klein2006,
    title = {Systematic multiscale models for deep convection on mesoscales},
    volume = {20},
    issn = {09354964},
    doi = {10.1007/s00162-006-0027-9},
    number = {5-6},
    journal = {Theoretical and Computational Fluid Dynamics},
    author = {Klein, Rupert and Majda, Andrew J.},
    year = {2006},
    pages = {525--551},
}

@book{MajdaAtmosphereOcean,
    title = {Introduction to {PDEs} and {Waves} for the {Atmosphere} and {Ocean}},
    isbn = {0-8218-2954-8},
    issn = {1529-9031},
    publisher = {American Mathematical society},
    author = {Majda, Andrew},
    year = {2003},
    note = {Series Title: Courant Lecture Notes in Mathematics 9},
}

@article{chengThreeScaleSingularLimits2018,
    title = {Three-{Scale} {Singular} {Limits} of {Evolutionary} {PDEs}},
    volume = {229},
    issn = {1432-0673},
    url = {https://doi.org/10.1007/s00205-018-1233-5},
    doi = {10.1007/s00205-018-1233-5},
    language = {en},
    number = {2},
    urldate = {2026-07-03},
    journal = {Archive for Rational Mechanics and Analysis},
    author = {Cheng, Bin and Ju, Qiangchang and Schochet, Steve},
    month = aug,
    year = {2018},
    pages = {601--625},
}

@article{schochetModeratelyFastThreeScale2020,
    title = {Moderately {Fast} {Three}-{Scale} {Singular} {Limits}},
    volume = {52},
    issn = {0036-1410},
    url = {https://epubs.siam.org/doi/10.1137/19M1287109},
    doi = {10.1137/19M1287109},
    number = {4},
    urldate = {2026-07-03},
    journal = {SIAM Journal on Mathematical Analysis},
    publisher = {Society for Industrial and Applied Mathematics},
    author = {Schochet, Steven and Xu, Xin},
    month = jan,
    year = {2020},
    pages = {3444--3462},
}

@article{Klein2010,
    title = {Regime of {Validity} of {Soundproof} {Atmospheric} {Flow} {Models}},
    volume = {67},
    issn = {1520-0469},
    url = {https://journals.ametsoc.org/jas/article/67/10/3226/104324/Regime-of-Validity-of-Soundproof-Atmospheric-Flow},
    doi = {10.1175/2010JAS3490.1},
    number = {10},
    journal = {Journal of the Atmospheric Sciences},
    author = {Klein, Rupert and Achatz, Ulrich and Bresch, Didier and Knio, Omar M. and Smolarkiewicz, Piotr K.},
    month = oct,
    year = {2010},
    pages = {3226--3237},
}

@article{breschSoundproofModelAcoustic2022,
    title = {The {Soundproof} {Model} of an {Acoustic}–internal {Waves} {System} with {Low} {Stratification}},
    volume = {24},
    issn = {1422-6928, 1422-6952},
    url = {https://link.springer.com/10.1007/s00021-022-00712-4},
    doi = {10.1007/s00021-022-00712-4},
    language = {en},
    number = {4},
    urldate = {2022-08-08},
    journal = {Journal of Mathematical Fluid Mechanics},
    author = {Bresch, Didier and Klein, Rupert and Liu, Xin},
    month = nov,
    year = {2022},
    pages = {95},
}

\end{document}